\documentclass[11pt]{amsart}
\usepackage{texmac}
\usepackage{semmac}

   \def\triv{{\mathbf 1}}
 
\def\H{H^1_{\text{\'et}}}
\let\Frob\Phi
\def\Xi{U}
\def\Cy#1{{\rm C}_{#1}}
\def\Gbar{\bar G}
\def\nr{{\text{nr}}}
\let\iso\simeq
\operators{codim}
\begin{document}

\title{Quotients of hyperelliptic curves and \'etale cohomology}

\author{Tim Dokchitser}
\address{Department of Mathematics, University Walk, Bristol BS8 1TW, UK}
\email{tim.dokchitser@bristol.ac.uk}

\author{Vladimir Dokchitser}
\address{Mathematics Institute, University of Warwick, Coventry CV4 7AL, UK}
\email{v.dokchitser@warwick.ac.uk}


\begin{abstract}
We study hyperelliptic curves $C$ with an action of an affine group of 
automorphisms $G$.
We establish a closed form expression for the quotient curve $C/G$ and 
for the first \'etale cohomology group of $C$ as a representation of $G$. 
%Both can be very simply expressed in terms of the $G$-action on the set of roots 
%of the defining equation, and the character $\gamma$. 
The motivation comes from the arithmetic of hyperelliptic curves over local fields, 
specifically
their local Galois representations and the associated invariants.
%: the conductor, local polynomial, root number etc. 
\end{abstract}

\maketitle

\tableofcontents

%\newpage

%%%%%%%%%%%%%%%%%%%%%%%%%%%%%%%
\section{Introduction}
%%%%%%%%%%%%%%%%%%%%%%%%%%%%%%%
\label{sintro}

%\marginpar{Let $\Phi\!:=\!\phi, \Psi\!:=\!\psi$?}
%\marginpar{Historical\\prewaffle?\\Former work either theoretical or algorithmic, ours is closed form}

%When $G$ is a finite group of automorphisms of an algebraic curve $C$,
%it acts on various vector spaces associated to $C$. Two classical examples 
%are the space of differentials and the \'etale cohomology.
%In both contexts there are 
%dimension formulae relating $C$ to $C/G$ 
%(Deuring-Shafarevich and Riemann-Hurwitz, respectively), 
%and the corresponding equivariant versions, that is descriptions 
%of the full $G$-representations on $C$ (due to \cite{?}, Weil \cite{Weil48} and 
%Serre \cite[Ch. VI, \S4]{SerLo}). 

This paper studies hyperelliptic curves $C/k$ 
with a group of affine automorphisms $G$. 
We show that both the equation of the quotient curve $C/G$ and
the $G$-representation $\H(C_{\bar k},\Q_l)$
%\'etale cohomology of $C$ with its $G$-action 
admit simple 
closed form descriptions in terms of the defining equation of $C$. The application 
we have in mind is to Galois representations of hyperelliptic curves over local 
fields~\cite{hyp}.

Throughout the paper $C$ will be a hyperelliptic curve over a 
field~$k$ with $\vchar k\ne 2$,
%\marginpar{$g=0,1$ allowed}
$$
  C: Y^2 = f(X)
$$
with $f(X)\!\in\!k[X]$ a squarefree polynomial of degree $\ge\!1$ that factors over~$\bar k$~as
$$
  f(X)=c \prod_{r\in R}(X-r), \qquad R\subset\bar k.
$$
%
%$$
%  C: Y^2 = c \prod_{r\in R}(X-r) \quad \qquad (R\subset\bar k),
%$$
Let $G\subset \Aut_k C$ be an affine group of automorphisms. Thus $g\in G$ acts as
$$
  g(X) = \alpha(g) X+\beta(g), \qquad g(Y)=\gamma(g)\,Y 
%  \qquad \qquad (g\in G).
$$
for some $\alpha(g), \beta(g), \gamma(g)\in k$.
%
%$$
%  
%$$
%
%We study groups of automorphisms $G\subset \Aut_k C$ preserving the divisor at $\infty$, 
%not necessarily tame. Thus, suppose $G$ acts via
%\marginpar{explain properly and fix notation for rest of paper}
%$$
%  X\mapsto g(X) = \alpha(g) X+\beta(g), \qquad Y\mapsto g(Y)=\gamma(g)\,Y \qquad \qquad (g\in G).
%$$
%
%Our aim is to give a closed form expression for the quotient curve $C/G$ and for the 
%\'etale cohomology group $\H(C_{\bar k},\Q_l)$ with $G$-action. 
In particular, $G$ acts naturally on the set of roots $R$ through the $X$-coordinate.
The punchline is that both the quotient curve $C/G$ and the 
\'etale cohomology group $\H(C_{\bar k},\Q_l)$ with $G$-action can  
be very simply expressed in terms of $\gamma$ and the $G$-action 
%(through the $X$-coordinate) 
on $R$.
Explicitly, writing $R/G$ for the set of orbits of $G$ on $R$, and $\C[R]$ for the 
permutation representation, we prove:

\begin{theorem}[Quotient curve]
\label{iappmain}
\noindent\par\noindent
\begin{enumerate}
\item
If $G$ contains the hyperelliptic involution,
%(equivalently, $\Syl_2 G$ is either not cyclic or $\Cy2$ acting 
%as $\{1,\iota\}$), 
then $C/G\iso \P^1_k$ .
\item 
  If $G$ acts trivially on the $Y$-coordinate, then
  $$
    C/G:\>\>y^2 = c\,(-1)^{|R|-|R/G|} \>
     {\prod_{O\in R/G}}
      (x-{\prod_{r\in O} r}).
  $$
\vspace{-10pt}
\item
  Otherwise, 
%  \marginpar{$\delta$ undefined}
  $$
    C/G:\>\>y^2 = c\>(-1)^{(m-1)(|R/G|-1)}\>(x-\mu) \prod_{O\in R/G\atop O\ne\Xi} 
      (x-\prod_{r\in O} r),
%      \begin{array}{llllll}
%      \cr
%      \delta=(m-1)(|R/G|-1)\cr
%      I_T=\prod_{g\in\ker\alpha}g(X)\cr
%      x=I\cr 
%      y=(I_T\!-\!\lambda)^{\lfloor m/2\rfloor}Y,\cr
%      \end{array}
  $$
\end{enumerate}
where%
\footnote{If $\vchar k\!=\!0$, then $|G|\!=\!m$, $|U|\!=\!1$; 
if $\vchar k\!=\!p>0$, then $|G|\!=\!m p^j$, $|U|\!=\!p^j$ and $p\!\nmid\!m$.}
$\Xi$ is the set of $a\in k$
that are fixed points of some non-trivial $g\in G$ acting on $k$ through 
the $X$-coordinate,
$m=|G|/|\Xi|$ and $\mu=\prod_{a\in \Xi}a^m$.
\end{theorem}

\begin{theorem}[\'Etale cohomology]
\label{iabmain}
%Let $k$ be a field of characteristic $\ne 2$, and $C/k$ a hyperelliptic curve
%$$
%  Y^2 = c\prod_{r\in R}(X-r), \qquad R\subset\bar k.
%$$
%Let $G\subset\Aut_k C$ be a group of automorphisms preserving the divisor at infinity; 
%in other words, every $g\in G$ acts via
%$$
%  X\mapsto g(X) = \alpha(g) X+\beta(g), \qquad Y\mapsto g(Y)=\gamma(g)\,Y.
%$$
For every prime $l\ne\vchar k$, and every embedding $\Q_l\injects\C$,
$$
  \H(C_{\bar k},\Q_l)\tensor\C \>\>\iso\>\> V\ominus\epsilon 
$$
as a complex representation of $G$, where
$$
  V = \tilde\gamma \otimes (\C[R] \ominus \triv), \qquad 
  \epsilon = \bigleftchoice{0}{\text{if $|R|$ is odd,}}{\det V}{\text{if $|R|$ is even,}}
$$
and $\tilde\gamma: G\to \C^\times$ is any one-dimensional representation 
with $\ker\tilde\gamma=\ker\gamma$.
The representation $\triv\oplus\epsilon$ is the permutation action 
of $G$ on the points at infinity of $C$.
\end{theorem}

The assumptions $\vchar k\ne 2$ and $l\ne \vchar k$ are certainly necessary.
When $\vchar k=2$, the equations for $C$ and the $G$-action need to be adjusted,
and the set of roots $R$ will no longer be the right object. 
Moreover, the actual answer has to be different: 
for example if $C/\bar\F_2$ is an elliptic curve with $G=Q_8$ or $G=\SL_2(\F_3)$, then
$\H(C_{\bar k},\Q_l)$ is not a 1-dimensional twist of a representation realisable over $\Q$.
Similarly, $\H(C_{\bar k},\Q_l)$ behaves differently when $l\!=\!\vchar k$: its
dimension drops, and it is not necessarily a rationally traced representation 
(e.g. for an ordinary elliptic curve over $\F_p$ with \hbox{$G\!=\!C_4$}).

%be different. 
%Say $E/\bar\F_2: y^2+y=x^3$,
%an elliptic curve with automorphism group $\SL_2(\F_3)$ of order 24.

\smallskip

Our motivation comes from the arithmetic of hyperelliptic curves over local fields, specifically
their $\H$ as a Galois representation and the associated invariants: the conductor, 
local polynomial, root number etc. Suppose for simplicity that $C/K$ is such a curve, 
and that it acquires good reduction over a finite Galois extension $F/K$. 
The inertia group $G=I_{F/K}$ and Frobenius element $\phi\in\GKK$ act naturally 
on the reduced curve $\bar C/k$, and $G$ acts as a group of affine automorphisms.
It turns out that $\H(C_{\bar K},\Q_l)$ as a $\GKK$-module is 
the same as $\H(\bar C_{\bar k},\Q_l)$ with its $G$- and $\phi$-action, 
and Theorem \ref{iabmain} describes the $G$-action explicitly. 
To complement this, 
we need to describe the action of Frobenius-like maps on the quotient curve $C/G$:

\bigskip

%By Bouw-Wewers, to compute local factor of a curve over a local field,  you need to 
%compute quotient curves and Frob action on them. 

\begin{theorem}
\label{iappfrob}
Suppose $\vchar k=p>2$, $G$ does not
%\marginpar{NOT semi-linear?}
contain the hyperelliptic involution, and
$\Frob: C\to C$ is a 
%semi-linear\footnote{i.e. $z\mapsto z^q$ on $k\subset k(C)$} 
morphism of the form 
$$ 
  \Frob\,X=aX^q+b,\quad \Frob\,Y=d\>Y^q  \qquad \qquad\text{($q$ power of $p$)}
$$
that normalises\footnote{meaning $\Phi G=G\Phi$ as sets} $G$.
Then $\Frob$ descends to a morphism \hbox{$\Psi\!:\! C/G\!\to\! C/G$}
given~by
%given in the notation of Theorem \ref{iappmain} by
$$
\begin{array}{llllll}
  \Psi\,x &=& \displaystyle a^{|G|} x^q + \prod_{g\in G} (\alpha(g)b+\beta(g)),\cr
  \Psi\,y &=& \displaystyle \bigleftchoice
     {d\>y^q}{\text{if $\gamma=\triv$,}}
     {a^{\lfloor m/2\rfloor |G|/m}\> d\>y^q}{\text{if $\gamma\ne\triv$,}}
\end{array}
$$
where $m$ is the prime-to-$p$ part of $|G|$, and the model for $C/G$ is that 
of Theorem \ref{iappmain}.
\end{theorem}

\noindent
See \S\ref{sex} for an example of how this theorem can be used over local fields to
determine the local Galois representation of a hyperelliptic curve
(and \cite{hyp} for the general theory).

\subsection*{Outline}

Section \ref{squo} concerns affine groups of automorphisms of hyperelliptic curves, 
the permutation action on the set of roots $R$ and the character $\gamma$. 
The main theorem here is Theorem \ref{appmain}, which is Theorem \ref{iappmain} 
with an explicit description of the quotient map.
Section \ref{setalequo} proves some general facts about $\H$ for quotient curves, 
in particular showing that $\H(C_{\bar k},\Q_l)$ is the unique representation $V$ of $G$ with rational character 
for which $\dim V^H\!=\!2\genus(C/H)$ for every subgroup $H\!<\!G$. 
Then \hbox{Section \ref{setalehyp}} applies 
this to hyperelliptic curves, proving Theorem \ref{iabmain} (=\ref{abmain}),
with a slightly cumbersome representation-theoretic computation.
An alternative approach would be to use the Weil--Serre description of the $G$-action on 
\'etale cohomology in terms of an equivariant Riemann--Hurwitz formula 
\cite{Weil48},
\cite[Ch.~VI,~\S4]{SerLo},
but that seems to be an equally long computation.
Section~\ref{sfrob} proves Theorem \ref{iappfrob} (=\ref{appfrob}). 
%In Section \ref{sex} we illustrate...

\begin{remark}
To obtain an explicit description for $\H(C_{\bar k},\Q_l)$, we only use that it is the unique 
representation $V$ of $G$ with rational character
for which $\dim V^H=2\genus(C/H)$ for every subgroup $H<G$ (Theorem \ref{abreconstr}).
Thus Theorem \ref{iabmain} would also hold for any other cohomology theory 
with these properties.
\end{remark}

\subsection*{Notation}
Throughout the paper, we use the following notation.

%\marginpar{genus $g$ conflicts with $g\in G$\\also $\zeta_m$ }
%\marginpar{$\Cy m\mapsto\backslash\text{Cy}_m$}
\begin{tabular}{llllll}
$C/k$ & hyperelliptic curve 
$Y^2 = c \prod_{r\in R}(X-r)$ with the right-hand\cr 
&side in $k[x]$, $R\subset\bar k$, $\vchar k\ne 2$;
genus $\le 1$ is allowed. \cr
$G$ & a finite group of affine automorphisms of $C/k$, acting by\cr
& $g(X) = \alpha(g) X+\beta(g),\>\> g(Y)=\gamma(g)\,Y$ for $g\in G$.\cr
$H^1(C)$ & $=\H(C_{\bar k},\Q_l)\otimes\C$ (choosing some embedding $\Q_l \injects \C$).\cr
$\triv, \oplus, \ominus$ & trivial representation, direct sum and direct difference.\cr
$\lfloor\cdot\rfloor$ & the floor function.\cr
$\zeta_m$ & primitive $m$th root of unity.\cr
\end{tabular}

%%%%%%%%%%%%%%%%%%%%%%%%%%%%%%%
\section{Quotients of hyperelliptic curves}
%%%%%%%%%%%%%%%%%%%%%%%%%%%%%%%
\label{squo}

We begin with affine actions and invariant functions on $\A^1$,
and then move to hyperelliptic curves.

\begin{lemma}
\label{autA1}
Let $k$ be a field, and $G\subset\Aut \A^1_k$ a finite group of affine linear transformations
$$
  X\mapsto g(X) = \alpha(g) X+\beta(g).
$$
Then
\begin{enumerate}
\item[(1)]
$G\!\iso\!T\!\rtimes\!\Cy m$, where $T$ is the subgroup of translations. 
If $\vchar k\!=\!0$, then $T$ is trivial.
If $\vchar k\!=\!p\!>\!0$, then $p\!\nmid\! m$, and $T$ is an elementary abelian $p$-group.
Moreover, $T$ is an $\F_p(\zeta_m)$-vector space, with
$\Cy m$ acting on $T$ by multiplication by $m$th roots of unity.
% and some generator of $\Cy m$ acts on $T$ by $v\mapsto \zeta_m v$.
%
%If $\vchar k=p>0$, then $G\iso T\rtimes A$ with $T$ an elementary abelian $p$-group acting 
%by translations and $A$ cyclic of order prime to $p$.
%\item[(1$^\prime$)]
%If $\vchar k=0$, then $G=A$ is cyclic (and $T$ is trivial).
\item[(2)]
Every $g\in G\setminus T$ has a unique fixed point $a\in \bar k$; moreover, $a\in k$.
\item[(3)]
If $m=1$, then %no element of $G$ has a fixed point on $\P^1$, and 
every $G$-orbit on $\A^1(\bar k)$ is regular%
\footnote{Recall that an action of a finite group $G$ on a set $X$ is \emph{regular} if 
$X\iso G$ as a $G$-set, equivalently the action is transitive and non-trivial elements of $G$
have no fixed points.}.
\item[(4)]
If $m\ne 1$, then there is a unique non-regular $G$-orbit on $\A^1(\bar k)$. 
It is a regular $T$-orbit, and consists of the fixed points of all $g\in G\backslash T$. 
\item[(5)]
The field of invariant rational functions $k(X)^G$ is generated by
$$
  I = \prod_{g\in G} (\alpha(g)X+\beta(g)),
%  I = \bigleftchoice{
%    ?
%  }{\text{ if $\,m=1$}}{
%    ?
%  }{\text{ if $\,m\ne 1$}}
$$
the unique $G$-invariant polynomial of degree $|G|$ with constant coefficient 0
and leading coefficient $(-1)^{|G|-1}$.
\item[(6)]
If $O\subset \bar k=\A^1(\bar k)$ is a regular $G$-orbit, then
$$
  \prod_{r\in O}(X-r) = (-1)^{|G|-1}\bigl(I-\prod_{r\in O} r\bigr).
$$
\item[(7)]
If $U\subset \bar k=\A^1(\bar k)$ is a non-regular $G$-orbit, 
let $I_T=\prod_{g\in T} g(X)$ and $\lambda=\prod_{r\in U}r$.
Then
$$
  \prod_{r\in U}(X-r) = I_T- \lambda,
$$
the unique monic 
polynomial of degree $|G|/m$ that is $G$-invariant up to scalars.
\end{enumerate}
\end{lemma}

\begin{proof}
(1)
The elements of $G$ with $\alpha(g)=1$ form a group 
of translations $T$, which is naturally an additive subgroup of $k$. As $G$ is finite, 
$T$ is trivial if $\vchar k=0$. 
The map $g\mapsto \alpha(g)$ embeds $A=G/T\injects k^\times$, so it is a cyclic group, $A\iso \Cy m$.

Suppose $\vchar k=p>0$. Then $p\nmid m$.
As $|T|$ and $|A|$ are coprime, the extension of $A$ by $T$ is split.
Take a generator $g$ of $A=\Cy m$ with $\alpha(g)=\zeta_m$. It conjugates a translation $x\mapsto x+v$ to
$x\mapsto x+\zeta_m v$. This makes $T$ into an $\F_p(\zeta_m)$-vector space, with the 
asserted conjugation action by $\Cy m$.

(2)
If $\alpha(g)\ne 1$, then
the map $X\mapsto \alpha(g)X+\beta(g)$ has a unique fixed point $a=\beta(g)/(1-\alpha(g))$ on $\A^1$, 
and it is $k$-rational. 

(3)
Clear, since translations have no fixed points.

(4)
Let $g$ be a generator of $\Cy m$ and $a\in k$ its fixed point.
The stabiliser of $a$ is precisely $\Cy m$ (as it meets $T$ trivially), so $a$ 
has orbit of length $|T|$; it is a regular $T$-orbit. 
Every other point in this orbit has a conjugate stabiliser
that meets $\Cy m$ trivially, by the uniqueness of fixed points. Thus, these stabilisers 
cover $(m-1)|T|$ elements of $G$, which must be the whole of $G\backslash T$.
Hence, the orbit of $a$ accounts for all fixed points of elements of $G$,
in other words every other $G$-orbit is regular. 
%
%Let $g$ be a generator of $\Cy m$ and $a\in k$ its fixed point.
%Every element of $G=T\rtimes \Cy m$ is of the form $tg^j$ for some 
%$t\in T$, $j\in\Z$, so the orbit of $a$ under $G$ is the same as under $T$. 
%As $t\cdot a$ is fixed by $t gt^{-1}$ and its powers (that span $G\backslash T$ as 
%$t$ and the power vary), 
%%\marginpar{explainify}
%%
%%  g P = P
%%  t g P = t P
%%  t g t^-1 t P = t P
%this orbit consists of exactly those points in $\A^1(\bar k)$ that are fixed by some 
%element of $G$. 
%Outside $G\cdot a$, every $g\in G$ has no fixed points, so every other $G$-orbit is 
%regular. 

(5) $I(X)$ is clearly $G$-invariant, and has the right degree $|G|$. Its constant coefficient is 0, and
the leading coefficient is
$$
  \prod_{g\in G}\alpha(g) \>=\> \Bigl(\prod_{i=0}^{m-1}\zeta_m^i\Bigr)^{|T|}\>=\>
  \leftchoice{(-1)^{|T|}=-1}{\text{if $2|m$}}{\quad\>1^{|T|}=1}{\text{if $2\!\nmid\! m$}}
     \>=\> (-1)^{|G|-1}.
$$

(6) The right-hand side is $G$-invariant. So is the left-hand side, since every $g\in G$ 
permutes its roots and $\alpha(g)^{|G|}=1$. The claim follows from (5), by
comparing the leading and the constant terms.

(7) Because $U$ is a $G$-orbit, the polynomial $\prod_{x\in U}(X-r)$ is 
$G$-invariant up to scalars. (And it is a unique such polynomial of this degree, 
because a non-regular orbit is unique if it exists.) By part (6) with $G=T$, 
$$
  \prod_{r\in U}(X-r) = I_T- \lambda.
$$
\end{proof}

\begin{proposition}
\label{autC}
Let $k$ be a field with $\vchar k\ne 2$, and $C/k$ a hyperelliptic curve 
$$
  C: Y^2 = c \prod_{r\in R}(X-r), \qquad R\subset \bar k.
$$
Let $G\subset\Aut_k C$ be a subgroup of affine automorphisms 
%preserving the divisor at infinity; in other words, every $g\in G$ acts via
$$
%  X\mapsto g(X) = \alpha(g) X+\beta(g), \qquad Y\mapsto g(Y)=\gamma(g)\,Y.
  g(X) = \alpha(g) X+\beta(g), \qquad g(Y)=\gamma(g)\,Y \qquad \qquad (g\in G).
$$
Write $K$ and $\Gbar$ for the kernel and the image of the map $X\mapsto g(X)$ from $G$ 
to $\Aut_k \A^1$.

\newpage

Then:

%
%Write $\Gbar$ for the image of $G$ in $\Aut_k \A^1$ via the action on the $X$-coordinate,
%and $K$
%
%Write $\iota$ for the hyperelliptic involution. 
%If $\iota\in G$, let $K\<G$ be $\langle\iota\rangle\iso \Cy2$, otherwise set $K=\{1\}$.
%Let $\Gbar=G/K$, the action of $G$ on the $X$-coordinate.
\begin{enumerate}
\item
$G$ is a central extension of $\Gbar$ by $K$, and
\par\noindent
\begin{tabular}{llllll}
\quad $G=\Gbar$   &$\iff$&  $|K|=1$  &  $\iff$ & hyperelliptic involution $\notin G$.\cr
\quad $G\ne\Gbar$ &$\iff$&  $|K|=2$  &  $\iff$ & hyperelliptic involution $\in G$.\cr
\end{tabular}
\item 
$\Gbar\iso T\rtimes_\alpha \Cy m$ with $T=\ker\alpha$ the subgroup of translations
of $\Gbar$.
%\item
%$G=\Gbar$ if and only if $G$ does not contain the hyperelliptic involution.
%, then either
%$G\iso T\rtimes_\alpha \Cy m$ or $G\iso T\rtimes_\alpha \Cy{2m}$.
\item
If $\vchar k=0$, then $T$ is trivial.
\item
If $\vchar k\!=\!p\!>\!0$, then $p\!\nmid\! m$, and $T$ is an elementary abelian \hbox{$p$-group}.
Moreover, $T$ is an $\F_p(\zeta_m)$-vector space, with
$\Cy m$ acting on $T$ by multiplication by $m$th roots of unity.
\item 
$\alpha$ is a primitive character of $\Gbar/T\iso \Cy m$, and $\gamma^2=\alpha^{|R|}$. 
The group $\Gbar$ acts naturally on $R$, and either
\begin{itemize}
\item[(a)]
$|R|\equiv 0\mod m$ and $\gamma^2=\triv$;
as a $\Gbar$-set, $R$ is a union of regular orbits; or
%\end{itemize}
%\noindent or
%\begin{itemize}
\item[(b)]
$|R|\equiv 1\mod m$ and $\gamma^2=\alpha\ne\triv$; 
as a $\Gbar$-set, $R$ is a union of regular orbits and
one non-regular orbit $R_0\iso \Gbar/\Cy m$, which is the set of fixed points of non-trivial 
elements of $\Gbar$.
\end{itemize}
\item If $|R|$ is even, then $\gamma(g)/\alpha(g)^{|R|/2}=\pm 1$ for $g\in G$, and 
it is $-1$ if and only if $g$ swaps the two points at infinity of $C$.
\end{enumerate}
\end{proposition}

\begin{proof}
(1) Clear.
(2),(3),(4) follow directly from Lemma \ref{autA1}.
%As $G$ has no hyperelliptic involution, $\alpha$ is faithful on $G$. Hence
%$\alpha: G/T\injects k^\times$, and its image is cyclic. It has order prime to $|T|$,
%so $G\iso T\rtimes (G/T)$ by Schur-Zassenhaus.
(5) Since automorphisms preserve the equation of $C$ up to a constant, 
$\Gbar$ permutes the elements of $R$, and $\gamma^2=\alpha^{|R|}$. 
The rest follows from the following lemma and Lemma \ref{autA1} (2),(3).
(6) The coordinates of the chart at infinity for $C$ are $u=1/X$ 
and $v=Y/X^{|R|/2}$, and the two points at infinity are 
$u=0,v=\pm\sqrt{c}$ on this chart, when $|R|$ is even. The claim follows. 
\end{proof}

\begin{lemma}
\label{groupstr}
Suppose $G$ is a group of the form $G=\F_p^r\rtimes \Cy m$ ($p\!\nmid\!m$)
such that no non-trivial elements $g\in \Cy m$ and $h\in\F_p^r$ commute.
%, 
%and every $1\ne g\in \Cy m$ 
%acts on $\F_p^r$ by conjugation with no fixed points. 
Let $R$ be a faithful $G$-set 
on which elements $\id\ne g\in\F_p^r$ have no fixed points, and 
elements $\id\ne g\in \Cy m$ have at most one fixed point. 
Then
\begin{enumerate}
\item 
$p^r\equiv 1\mod m$;
\item
$R$ is a union of $G$-regular orbits plus at most one non-regular
orbit $R_0\iso G/\Cy m$, consisting of fixed points of non-trivial 
elements of $G$ on $R$.
\end{enumerate}
\end{lemma}

\begin{proof}

(1)
From the non-commutativity assumption, it follows that the 
orbits of $\Cy m$ on $\F_p^r\setminus\{\id\}$ have length $m$. 

(2)
Since non-trivial elements of $\F_p^r$ have no fixed points,
%First note that 
% that non-trivial elements of $\F_p^r$ have at most one point on $R$ implies that
$R$ is a union of regular $\F_p^r$-orbits. 
Now suppose that $R$ has a non-regular orbit of $G$.  
It is a union of $m/k$ regular $\F_p^r$-orbits for some $k\!>\!1$, 
which $\Cy m$ permutes transitively. 
As $k$ and $|\F_p^r|$ are coprime, elements of $C_k$ must have fixed points in each one of them.
But they have at most one fixed point in total, and so $k=m$, such an orbit
is $\iso  G/\Cy m$, and it is necessarily unique
(and clearly consists of fixed points of non-trivial elements of $G$).
Hence
$$
  R \iso  G \amalg \cdots \amalg  G 
    \qquad\text{or}\qquad
  R \iso  G \amalg \cdots \amalg  G \amalg  G/\Cy m.
$$
In the first case $|R|\!\equiv 0\!\mod m$,
and in the second case $|R|\!\equiv \!p^r\!\equiv \!1\!\mod m$.
%
%. In the second case, $R$ is (by the above argument
%with $ G=\Cy m$) a union of regular orbits plus a singleton, and so $|R|\equiv 1\mod m$.
%Finally, $\gamma^2=\alpha^{|R|}$ by assumption, which is $\triv$ or $\alpha$, respectively.
\end{proof}

%\newpage

\begin{theorem}
\label{appmain}
Let $C$ be a hyperelliptic curve over a field~$k$ with $\vchar k\ne 2$,
$$
  C: Y^2 = c f(X); \qquad  
    f(X) = \prod_{r\in R}(X-r),
$$
and $G\subset \Aut_k C$ a group of automorphisms acting via
$$
  g(X) = \alpha(g) X+\beta(g), \qquad g(Y)=\gamma(g)\,Y \qquad \qquad (g\in G).
$$
%(i.e. preserving the divisor at $\infty$). 
Let $I=\prod_{g\in G} g(X)$, and write $R/G$
%\marginpar{$I\mapsto I(X)?\\J\mapsto J(X)?$} 
for the set of orbits of $G$ on the roots of $f$ through the $X$-coordinate action.
\begin{enumerate}
\item
If $G$ contains the hyperelliptic involution 
%(equivalently, $\Syl_2 G$ is either not cyclic or $\Cy2$ acting 
%as $\{1,\iota\}$), 
then $C/G\iso \P^1_k$ with field of rational 
functions \hbox{$k(C/G)=k(\sqrt I)$}.
\item 
  If $G$ acts trivially on the $Y$-coordinate, then
  $$
    C/G:\>\>y^2 = c\, (-1)^{|R|-|R/G|} 
     {\prod_{O\in R/G}}
      (x-{\prod_{r\in O} r}) 
%      \qquad \quad 
%      \begin{array}{llllll}
%      x=I\cr y=Y.\cr
%      \end{array}
%      \bigleftchoice{y=Y}{}{x=I}.
  $$
with $x=I$ and $y=Y$.
%\item 
%  If $G$ acts on the $Y$-coordinate through a non-trivial group of odd order, then
%  $R$ has a non-regular $G$-orbit $R_0$, and
%  $$
%    C/G:\>\>Y^2 = c\, (-1)^{|R|-|R/G|-1} *
%     {\prod_{O\in (R\setminus R_0)/G}}
%      (X-{\prod_{r\in O} r}).
%  $$
%\item 
%  If $\iota\notin G$ and $G$ acts on the $Y$-coordinate through a group of even order, then
%  \marginpar{$\lambda^2?$}
%  $$
%    C/G:\>\>Y^2 = c (X-\lambda^2) \prod_{O\in R/G} 
%      (X-\prod_{r\in O} r). \qquad \qquad \quad 
%  $$
%Here $\lambda=\prod_{a\in\Sigma'} a^n$, where $\Sigma'$ is the set of $a\in k$
%that are fixed points of some non-trivial $g\in G$, and $n=|G|/2|\Sigma'|$.
\item
  Otherwise, 
  $$
    \qquad 
    C/G:\>\>y^2 = c\>(-1)^\delta\>(x-\lambda^m) \prod_{O\in R/G\atop O\ne \Xi} 
      (x-\prod_{r\in O} r) 
%      \quad 
%      \begin{array}{llllll}
%      \lambda=\prod_{a\in \Xi}a\cr
%      \delta=(m-1)(|R/G|-1)\cr
%      I_T=\prod_{g\in\ker\alpha}g(X)\cr
%      x=I\cr
%      y=(I_T\!-\!\lambda)^{\lfloor m/2\rfloor}Y,\cr
%      \end{array}
  $$
with $x\!=\!I$ and $y\!=\!(I_T\!-\!\lambda)^{\lfloor m/2\rfloor}Y$.
Here $\Xi$ is the set of $a\!\in\! k$~
that are fixed points of some non-trivial $g\in G$ acting through the $X$-coordinate, 
$$      
  \qquad \qquad  
  m=\frac{|G|}{|\Xi|}, \quad 
  \lambda=\prod_{a\in \Xi}a, \quad \delta=(m-1)(|R/G|-1), \quad I_T=\prod_{g\in\ker\alpha}g(X).
$$
%and $m=|G|/|\Xi|$.
\end{enumerate}
\end{theorem}

\begin{proof}
%Recall from Lemma \ref{autA1}(5) 
%that $I$ is the unique $G$-invariant polynomial in $X$ of degree $|G|$ 
%with constant coefficient 0 and leading coefficient $(-1)^{|G|-1}$.
%In particular, 
%if $O\subset k$ is a regular $G$-orbit, via the action of $G$ on the $X$-coordinate, then
%\begin{equation}\label{rorbit}
%  \prod_{r\in O}(X-r) = (-1)^{|G|-1}\bigl(I-\prod_{r\in O} r\bigr),
%\end{equation}
%as can be seen by comparing the leading and the constant terms.
(1) 
Write $\iota$ for the hyperelliptic involution. We clearly have
\hbox{$C/\langle \iota\rangle\!\iso\!\P^1_k$}, and $k(C/\langle \iota\rangle)=k(X)$. 
Its quotient by $G/\langle\iota\rangle$ is again $\P^1_k$, and its field of rational functions is
$k(\sqrt I)$ by Lemma \ref{autA1}(5).

(2) The polynomials $Y$ and $I$ are $G$-invariant, and 
$$
  k(C/G)=k(C)^G=k(I,Y)
$$
by degree considerations. 
To find the relation between $I$ and $Y$, first note that as $\gamma^2=\triv$, 
$R$ is a union of regular $\Gbar$-orbits by Proposition \ref{autC}.
%(case (5b) cannot occur as $\gamma^2=\triv$).
Because $Y^2/c=f(X)\in k(X)$ is $G$-invariant, by Lemma \ref{autA1}(6),
$$
\begin{array}{llllll} 
  f(X)&=&\displaystyle \prod_{r\in R}(X-r) = \>\prod_{O\in R/G} (-1)^{|G|-1}\bigl(I-\prod_{r\in O} r\bigr)\cr
      &=&\displaystyle \>(-1)^{|R|-|R/G|}\>\prod_{O\in R/G} \bigl(I-\prod_{r\in O} r\bigr),\cr
\end{array}
$$
which gives the required identity between $x=I$ and $y=Y$.

(3) Write $G=T\rtimes \Cy m$ as in Proposition \ref{autC}(1,2).
First, suppose that $T$ is trivial, so that $G=\Cy m$ with $m>1$. We have two cases:

{\bf Case $|T|\!=\!1$, $m$ odd.} 
%\marginpar{beautify\\boldface?}
Since $\gamma\ne \triv$, we are in case (5b) of Proposition~\ref{autC}. 
So $\gamma=\alpha^{-(m-1)/2}$ (the unique character that squares to $\gamma^2=\alpha$)
and there is a unique $G$-invariant root $\lambda\in R$. Hence
elements $g\in G$ must act on the curve $C$ as
$$
  g(X) = \alpha(g) (X-\lambda)+\lambda, \qquad g(Y)=\alpha(g)^{-(m-1)/2}\,Y,
$$
and the polynomials
$$
  I=(X-\lambda)^m+\lambda^m, \qquad J=(X-\lambda)^{(m-1)/2}Y
$$
are easily seen to be $G$-invariant (and the first one is $I$ by Lemma \ref{autA1}(5)). 
As $[k(X):k(I)]\le m$ and $k(X,Y)=k(X,J)$, the polynomials $I$ and $J$ 
generate $k(C)^G$ by degree considerations. They satisfy the relation
$$
%\begin{equation}\label{IJ1}
  J^2 = Y^2(X-\lambda)^{m-1} = c\>f(X)\>\frac{I-\lambda^m}{X-\lambda} 
      = c\>(I-\lambda^m)\>\frac{f(X)}{X-\lambda}.
%\end{equation}
$$
Finally, since $f(X)$ has all roots but $\lambda$ coming in $G$-regular orbits, 
by Lemma~\ref{autA1}(6), 
$$
%\begin{equation}\label{IJ2}
  \frac{f(X)}{X-\lambda} = {\prod_{O\in R/G\atop O\ne\{\lambda\}}}
      (I-{\prod_{r\in O} r}),
%\end{equation}
$$
which gives the required equation 
%\eqref{IJ1},\eqref{IJ2} give the 
$$
  J^2 = c (I-\lambda^m) \prod_{O\in R/G\atop O\ne\{\lambda\}}
      (I-{\prod_{r\in O} r}).
$$

{\bf Case $|T|\!=\!1$, $m$ even.} 
If we were in case (5b) of Proposition \ref{autC}, then 
$\alpha$ would have order $m$ and $\gamma$ order $2m$; in that case, $\ker\alpha\ne\{\id\}$ 
and $G$ contains the hyperelliptic involution --- a contradiction.
%
%Some $g\in G$ then acts as 
%$X\mapsto X+\beta, Y\mapsto -Y$, and its $p$th power is the hyperelliptic involution, 
%contradiction. (If $k$ has characteristic zero, $\beta$ must already be $0$.)
Hence we must be in case (5a), in other words $R$ is a union of regular
$\Gbar$-orbits.
Note that $\gamma=\alpha^{m/2}$ since $\gamma^2=\triv$ and $\gamma\ne\triv$.

By Lemma \ref{autA1}(3), there is a unique fixed point $\lambda\in k$ of $G=\Cy m$,
with $\lambda\notin R$ by the above. The polynomials
$$
  I=-(X-\lambda)^m+\lambda^m, \qquad J=(X-\lambda)^{m/2}Y
$$
are $G$-invariant (and the first one is $I$ by Lemma \ref{autA1}(5)), and generate $k(C)^G$
by degree considerations.
Using Lemma \ref{autA1}(6) as before, we get the required relation 
$$
\begin{array}{llllll}
  J^2 &=& \displaystyle (X-\lambda)^m\>c\>f(x) =  c\>(-I+\lambda^m)\>\prod_{O\in R/G} -(I-{\prod_{r\in O} r})\cr
      &=& \displaystyle (-1)^{1+|R/G|}\>c\>(I-\lambda^m)\>\prod_{O\in R/G} (I-{\prod_{r\in O} r}).\cr
\end{array}
$$

{\bf Case $|T|\!\ne\! 1$.} 
Finally, suppose that $T$ is non-trivial. First, we can compute $C/T$ by using (2)
and then $C/G=(C/T)/\Cy m$ using the $|T|=1$ cases. 
We have
$$
  k(C/T)=k(C)^T = k(Y, I_T), \qquad I_T = \prod_{g\in T} g(X),
$$
and the equation for $C/T$ is
$$
  C/T:\>\>Y^2 = c \prod_{O\in R/T} (I_T-{\prod_{r\in O} r}).
$$
The group $\Cy m$ acts on this curve by affine transformations of the form
$$
  g(I_T) = \alpha_T(g)I_T+\beta_T(g), \quad g(Y)=\gamma(g)Y \qquad\quad (g\in \Cy m),
$$
and
$$
  \alpha_T(g)=\alpha(g)^{|T|}=\alpha(g),
$$
since $|T|\equiv1\mod m$ by Lemma \ref{groupstr}(1) and Proposition \ref{autC}(5).
In particular, $\Cy m$ acts non-trivially on the $Y$-coordinate of $C/T$ and does not contain the
hyperelliptic involution of $C/T$.  
By Lemma \ref{autA1} (7), $I_T-\lambda$ is the unique monic polynomial of degree
$|G|/m$ that is $G$-invariant up to scalars.
In other words, 
%$
%  \lambda = \prod_{a\in \Xi}a
%$.
on the $I_T$-coordinate $\lambda$ is the unique fixed point for the $\Cy m$-action.
% has a unique fixed point 
%$\lambda$, 
%and its preimage under $C\mapsto C/T$
%is the unique non-regular $G$-orbit $U\subset\bar k$. It is a regular $T$-orbit
%by Lemma \ref{autA1} (3) with $G=T$, and
%
%which can be expressed as 
%$
%  \lambda = \prod_{a\in \Xi}a
%$.
%\vskip 3cm
%
%
%
%
%The unique fixed point $\lambda\in k$ under the $\Cy m$ action on the 
%
%The $G$-action on the $X$-coordinate
%of $C$ has a unique non-regular orbit $\Xi$, and it descends to a unique point on $C/T$ 
%whose $I_T$-coordinate 
%is invariant under the action of $\Cy m$.
%
Hence, by the $|T|=1$ cases we find the quotient
$$
  (C/T)/\Cy m: \>\> 
   \>y^2 = c\>(-1)^{\delta}\>(x-\lambda^m) \prod_{O\in (R/T)/\Cy m\atop O\ne \{\lambda\}} 
      (x-\prod_{r\in O} r),
$$
where 
$$
\begin{array}{llllll}
  \delta&=&(m-1)(|(R/T)/\Cy m|-1) = (m-1)(|R/G|-1),\\[3pt]
  x&=& \displaystyle \prod_{g\in \Cy m}g(I_T)=\prod_{g\in \Cy m}g\Bigl(\prod_{t\in T}t(X)\Bigr)=I,\\[2pt]
  y&=& (I_T-\lambda)^{\lfloor m/2\rfloor}Y,\cr
\end{array}
$$
using that $R$ is a union of regular $T$-orbits by Proposition \ref{autC} (5) with $G=T$.
\end{proof}

%\begin{corollary}
%$$
%\begin{array}{|c|c|c|}
%\hline
%\text{Genus of $C/G$}         &   n=rmp^k   & n=(rm+1)p^k \cr
%\hline
%\iota\in G                    &   0         & 0           \cr
%\hline
%\iota\notin G, g\cdot Y\ne Y  &   ?         & ?           \cr
%\text{for some } g\in G       &&\cr
%\hline 
%\iota\notin G, g\cdot Y=Y     &   ?         & ?           \cr
%\text{for all } g\in G        &&\cr
%\hline
%\end{array} 
%$$
%\end{corollary}

%\begin{corollary}
%Suppose $C$ and $G$ are as in Theorem \ref{appmain}, and $|G|=mp^k$ with $(m,p)=1$. 
%The genus of $C/G$ is 
%$$
%\begin{array}{ll}
%%\hline
%\text{Case} & \text{Genus of $C/G$} \cr
%\hline
%\iota\in G & 0 \cr 
%\iota\notin G,\>\> n=rp^km, \ref{appmain}(2)   & \lfloor (r-1)/2\rfloor \cr
%\iota\notin G,\>\> n=rp^km, \ref{appmain}(3)   & \lfloor r/2\rfloor \cr
%\iota\notin G,\>\> n=rp^km+p^k                 & \lfloor r/2\rfloor \cr
%%\hline
%\end{array}
%$$
%\end{corollary}

%\marginpar{define $\iota$\\globally?}

\begin{corollary}
\label{genuscor}
Suppose $C$ and $G$ are as in Theorem \ref{appmain}, and 
%$|G|=mp^k$ with $(m,p)=1$. 
let $r$ be the number of regular orbits of $G$ on $R$. Then
$$
  \genus(C/G)=\left\{
    \begin{array}{llllll}
    0 &&&& \text{if } \iota\in G,\cr
    \bigl\lfloor \frac r2-\frac 12 \bigl\rfloor \!\!&=\!\!&
    \bigl\lfloor \frac{|R|}{2|G|}-\frac 12\bigr\rfloor && \text{if } \gamma=\triv, \cr
    \bigl\lfloor \frac r2 \bigl\rfloor \!\!&=\!\!&
    \bigl\lfloor \frac{|R|}{2|G|} \bigr\rfloor &&\text{if } \iota\notin G, \gamma\ne\triv. \cr
    \end{array}
\right.
$$
%Old table: the genus of $C/G$ is 
%$$
%\begin{array}{llll}
%\text{Case} & \text{Genus of $C/G$} \cr
%\hline
%\iota\in G                   & 0 \cr 
%\gamma=\triv                     & \lfloor (r-1)/2\rfloor &=& \lfloor (|R|/|G|-1)/2\rfloor \cr
%\iota\notin G, \gamma\ne\triv   & \lfloor r/2\rfloor     &=& \lfloor |R|/|G|/2\rfloor \cr
%\end{array}
%$$
%where $r$ is the number of regular orbits of $G$ on $R$. 
%\marginpar{If we prefer, in the second line we could 
%also say $\gamma$ has odd order, i.e. $-1$ not in the image of $\gamma$, and in the third
%$\gamma$ has even order, i.e. $-1$ in the image of $\gamma$.}
\end{corollary}

\begin{proof}
Apply Theorem \ref{appmain}, and recall that the genus of a hyperelliptic curve given
by a polynomial of degree $n$ is $\lfloor \frac {n-1}2\rfloor$.
\end{proof}

%\begin{corollary}
%Suppose $C$ and $G$ are as in Theorem \ref{appmain}, and $|G|=mp^k$ with $(m,p)=1$. 
%The genus of $C/G$ is 
%$$
%\begin{array}{lll}
%\text{Case} & \text{Genus of $C/G$} \cr
%\hline
%\iota\in G                             & 0 \cr 
%\gamma=\triv                               & \lfloor (|R|/G-1)/2\rfloor \cr
%\iota\notin G, \gamma\ne\triv, X=X_r      & \lfloor |R|/G/2\rfloor \cr
%\iota\notin G, \gamma\ne\triv, X\ne X_r   & \lfloor (|R|-p^k)/G/2\rfloor & = \lfloor |R|/G/2\rfloor \cr
%\end{array}
%$$
%\end{corollary}

%%%%%%%%%%%%%%%%%%%%%%%%%%%%%%%%%%%%%%%%%%%%%%%
\section{\'Etale cohomology of quotient curves}
%%%%%%%%%%%%%%%%%%%%%%%%%%%%%%%%%%%%%%%%%%%%%%%
\label{setalequo}

In this section, $C/k$ is any non-singular projective curve with a finite group of 
automorphisms $G<\Aut_k C$. We will show that the $G$-action on the \'etale cohomology group
$\H(C_{\bar k},\Q_l)$ is determined by the genera of the quotients of $C$ by subgroups of $G$.

Recall that for every prime $l$ there are canonical isomorphisms
%\marginpar{also $l=p$?}
%\marginpar{SGA 7?}
$$
  \H(C_{\bar k},\Q_l) \iso \H(\Jac(C_{\bar k}),\Q_l) \iso (V_l\Jac(C))^*,
$$
where, as usual,
$$
  V_l \Jac(C) = (\varprojlim \Jac(C)[l^n])\tensor_{\Z_l} \Q_l 
$$
is the vector space associated to the Tate module, and $*$ is $\Q_l$-linear dual. 
If~$G<\Aut_k C$ is a group of automorphisms, these become naturally 
$\Q_l G$-representations, and the isomorphisms respect this structure.

\begin{theorem}
\label{thmcover}
Let $\pi: C\to D$ be a Galois cover of non-singular projective curves over a field~$k$, 
with Galois group $G$, and let $l$ be a prime number.
\begin{enumerate}
\item
%$V_l\Jac C$ and 
If $l\ne\vchar k$, then
$\H(C_{\bar k},\Q_l)$ is a $\Q_lG$-representation with rational character.
\item 
There is an isomorphism of $\Q_l$-vector spaces
$$
%\begin{equation}\label{vlc}
%  (V_l\Jac C)^G\iso V_l(\Jac D), \qquad 
  \H(C_{\bar k},\Q_l)^G\iso \H(D_{\bar k},\Q_l).
%\end{equation}
$$
\item
If $\Phi: C\to C$ is a morphism of curves that commutes with $\pi$, then
the isomorphism in (2) %\eqref{vlc} 
commutes with the action on $\Phi$. 
\item
If $l\ne \vchar k$, then $\genus(D)=\frac 12\dim \H(C_{\bar k},\Q_l)^G$.
\end{enumerate}
\end{theorem}

\begin{proof}
(1)
By the Lefschetz fixed-point formula \cite[Thm. 25.1]{MilE}, 
the trace of any $\sigma\in G$ on $H^1$ is an integer, 
namely 2 minus the number of fixed points of $\sigma$ on $C$.

(2)
The projection $\pi$ induces pushforward and pullback on divisors,
$$
  \pi_*: \Div(C)\lar \Div(D), \qquad \pi^*: \Div(D)\lar \Div(C).
$$
%and $\pi_*\pi^*$ is multiplication by $|G|$, while $\pi^*\pi_*$ is the trace map \hbox{$\tr: P\mapsto \sum_{g\in G} P^g$}.
The composition $\pi_*\pi^*$ is multiplication by $|G|$, while $\pi^*\pi_*$ is the trace map~\hbox{$\tr: P\mapsto \sum_{g\in G} P^g$}.
These maps descend to $\Pic^0(C)=\Jac(C)$ and its $l^n$-torsion $\Jac(C)[l^n]$. 
The image of $\Tr$ on $V_l\Jac(C)$ is the group of $G$-invariants, and
$$
  (V_l\Jac(C))^G \iso \tr (V_l\Jac(C)) \iso \pi_*(V_l\Jac(C)) \iso V_l\Jac(D),
$$
as multiplication by $|G|$ is an isomorphism on $V_l$, and 
hence $\pi^*$ is injective and $\pi_*$ is surjective.

(3)
Clear from the construction in (2).

(4)
Clear from (2) and the genus formula $\dim \H(D_{\bar k},\Q_l)=2\genus(D)$.

\end{proof}

\begin{lemma}
\label{uniqinv}
Let $G$ be a finite group. A representation $V$ of $G$ with rational character is 
uniquely determined by $\dim V^H$ for all cyclic subgroups $H<G$.
\end{lemma}

\begin{proof}
A suitable multiple of $V$ is realisable over $\Q$ (\cite{SerLi} Ch. 12, Prop. 34), 
and for such representations the claim is proved in 
\cite{SerLi} Ch. 13, Cor. to Thm.~30$\,'$.
\end{proof}

\begin{theorem}
\label{abreconstr}
Let $C/k$ be a curve, and $l\ne \vchar k$ a prime number. Suppose 
$G$ is a finite group acting as automorphisms of $C/k$ (not necessarily faithfully). Then
%\subset\Aut(C/k)$. 
$\H(C_{\bar k},\Q_l)$ is the unique $\Q_l$-representation $V$ of~$G$~such~that
\begin{itemize}
\item[(a)] all character values of $G$ on $V$ are rational, and 
\item[(b)] $\dim V^H=2\genus(C/H)$ for every subgroup $H<G$.
\end{itemize}
\end{theorem}

\begin{proof}
Uniqueness follows from Lemma \ref{uniqinv}.
Theorem \ref{thmcover}(1,4) shows that $V=\H(C_{\bar k},\Q_l)$ satisfies (a) and (b)
when the action is faithful. 
Therefore, if $q: G\to\Aut C$ is any action, then 
$$
  \dim V^H = \dim V^{q(H)} = 2\genus(C/q(H)) = 2\genus(C/H),
$$
and $q(g)$ has rational trace on $V$ for every $g\in G$.
%To see that $\H(Z,\Q_l)$ is a rationally traced representation, either use Serre's Local Fields or 
%observe that traces of elements are given by the Lefschetz trace formula (and are therefore
%integers, and are independent of $l$). Next use that $H$-invariants on $\H(Z,\Q_l)$ are
%exactly $H^1$ of the quotient curve.
\end{proof}

\begin{remark}
\label{remcffield}
In Theorem \ref{abreconstr}(b), `subgroup' may be replaced by `cyclic subgroup',
since these are sufficient for Lemma \ref{uniqinv}.
Note also that the assertion concerns only the character of $V$, so the result remains true
after any extension of the coefficient field (e.g. to $\bar\Q_l$ or $\C$).
%\item
%Note that \'etale cohomology may be 
%\end{itemize}
\end{remark}

\medskip

%%%%%%%%%%%%%%%%%%%%%%%%%%%%%%%%%%%%%%%%%%%%%%%%%%%%
\section{\'Etale cohomology of hyperelliptic curves}
%%%%%%%%%%%%%%%%%%%%%%%%%%%%%%%%%%%%%%%%%%%%%%%%%%%%
\label{setalehyp}

In this section we prove the following theorem that describes the action of automorphisms
on the first \'etale cohomology group of a hyperelliptic curve.

\begin{theorem}
\label{abmain}
Let $k$ be a field of characteristic $\ne 2$, and $C/k$ a hyperelliptic curve given by
$$
  Y^2 = c\prod_{r\in R}(X-r), \qquad R\subset\bar k.
$$
Let $G\subset\Aut_k C$ be an affine group of automorphisms acting as
$$
  g(X) = \alpha(g) X+\beta(g), \qquad g(Y)=\gamma(g)\,Y \qquad \qquad (g\in G).
$$
Then for every prime $l\ne\vchar k$ and every embedding $\Q_l\injects\C$,
$$
  \H(C_{\bar k},\Q_l)\tensor\C \>\>\iso\>\> V\ominus\epsilon 
$$
as a complex representation of $G$, where
$$
  V = \tilde\gamma \otimes (\C[R] \ominus \triv), \qquad 
  \epsilon = \bigleftchoice{0}{\text{if $|R|$ is odd,}}{\det V}{\text{if $|R|$ is even,}}
$$
and $\tilde\gamma: G\to \C^\times$ is any one-dimensional representation 
with $\ker\tilde\gamma=\ker\gamma$.
The representation $\triv\oplus\epsilon$ is the permutation action 
of $G$ on the points at infinity of $C$.
\end{theorem}

%\begin{remark}
%The character $\gamma$ factors through a cyclic group, which is the largest quotient of $G$ 
%of order prime to $\vchar k$ (see Lemma \ref{autA1}). 
%\end{remark}
%
\begin{proof}
By Theorem \ref{abreconstr} (and Remark \ref{remcffield}), 
it suffices to prove that $V\ominus\epsilon$ has rational
character and that $\dim (V\ominus\epsilon)^H=2\genus(C/H)$ for every $H<G$.
This is a purely representation-theoretic statement, proved below in Theorem \ref{abrep}.
By Lemma \ref{autA1} and Proposition \ref{autC}, the theorem applies here
with $G$, $R$,
% with the natural action of $G$ through the $X$-coordinate, 
$\tilde\gamma$ and with $\kappa$ the hyperelliptic involution if it is in $G$
and $\kappa=\id$ otherwise.
The theorem shows that $V\ominus\epsilon$ has rational character, and that
$$
\dim(V\ominus\epsilon)^H=\left\{
\begin{array}{llllll}
0 & \text{if } \text{hyperelliptic involution}\in H \cr
2\bigl\lfloor \frac{|R|}{2|H|}-\frac 12\bigr\rfloor & \text{if } \gamma(H)=1 \cr
2\bigl\lfloor \frac{|R|}{2|H|} \bigr\rfloor & \text{otherwise.} \cr
\end{array}
\right.
$$
By Corollary \ref{genuscor}, this is precisely $2\genus(C/H)$. 

The claim about $\triv\oplus\epsilon$ is clear when $|R|$ is odd (as there is one point of
infinity, fixed by $G$), and follows from Proposition \ref{autC} (6) and 
Lemma \ref{EGA} when $|R|$ is even. % in all other cases
\end{proof}

%\newpage

\begin{theorem}
\label{abrep}
Let $G$ be a finite group acting on a set $R$, and $\gamma: G\to\C^\times$ a one-dimensional 
representation, satisfying the following
\begin{itemize}
\item
The kernel of the action of $G$ on $R$ is generated by an element $\kappa$ of order 1 or 2.
\item 
Write $\Gbar=G/\langle \kappa\rangle$. Then
$\Gbar\iso T\rtimes \Cy m$, where $T$ is an $\F_p(\zeta_m)$-vector space with $p\nmid m$,
and some generator of $\Cy m$ acts on $T$ by $v\mapsto \zeta_m v$.
%\item
%Let $\kappa$ be the generator of the kernel of this action (of order 1 or 2).
%It acts trivially on $R$, and $G/\langle\kappa\rangle$ acts 
%faithfully. 
\item
Every $g\in T\setminus\{\id\}$ has no fixed points on $R$, and every $g\in G\setminus\{\id,\kappa\}$ 
has at most one fixed point.
\item
If $\kappa\ne\id$ then $\gamma(\kappa)=-1$.
\item
$\gamma$ is trivial on $T$, and $\gamma^2$ is an $|R|$th power of some 
one-dimensional representation $\alpha$ of $\Cy m=\Gbar/T$ of exact order $m$.
\end{itemize}
For $H<G$ define 
$$
f(H)=\left\{
\begin{array}{llllll}
%0 & \text{if} & \kappa\in H, \gamma(\kappa)=-1 && \text{\rm (Case \I)}\cr
0 & \text{if }  \id\ne \kappa\in H && \text{\rm (Case \I)}\cr
2\bigl\lfloor \frac{|R|}{2|H|}-\frac 12\bigr\rfloor & \text{if } \gamma(H)=1 && \text{\rm (Case \II)}\cr
2\bigl\lfloor \frac{|R|}{2|H|} \bigr\rfloor & \text{otherwise}
%& \text{if} & \kappa\notin H, \gamma(H)\ne 1 
&& \text{\rm (Case \III),}\cr
\end{array}
\right.
$$
%$$
%\begin{array}{lllll}
%\text{Case} && f(H) \cr   % Zase and Henus
%\hline
%\vphantom{\int^R}
%\I & \kappa\in H                   & 0 \\[1pt]
%\II & \gamma(H)=1                     & 2\bigl\lfloor \frac{|R|}{2|H|}-\frac 12\bigr\rfloor \\[4pt]
%\III & \kappa\notin H, \gamma(H)\ne 1   & 2\bigl\lfloor \frac{|R|}{2|H|} \bigr\rfloor \cr
%\end{array}
%$$
and let
$$
%   , \qquad\text{where}\qquad
   V = \gamma \otimes (\C[R] \ominus \triv), \qquad 
   \epsilon=\bigleftchoice 0{\text{if $|R|$ is odd,}}{\det V}{\text{if $|R|$ is even.}}
$$
Then $V\ominus\epsilon$ is the unique representation with rational character for which
$$
%\begin{equation}\label{eqdim}
  \dim (V\ominus\epsilon)^H = f(H) \qquad \text{for all $H<G$}.
  \eqno(\ast)
%\end{equation}
$$
\end{theorem}

\begin{proof}
Lemma \ref{uniqinv} proves uniqueness.
%
%shows that if it satisfies $(\ast)$, it is necessarily unique. 
By Lemma \ref{abR2}(2), $V\ominus\epsilon$ has rational character, and it remains to 
prove $(\ast)$. 
When $H=G$, this is shown in Lemmas \ref{lemcaseI}, \ref{lemcaseII}, \ref{lemcaseIII}, in Case \I, \II{} 
and \III, respectively.
To establish $(\ast)$ for general $H<G$, we can invoke the lemmas with $H$ in place for $G$ 
and with $R$ and $\gamma$ restricted to $H$. 
(Observe that every $H<G$ satisfies the conditions of the theorem.)
%
%
%Observe that every subgroup $H<G$ satisfies the conditions of the theorem 
%(with $R$ and $\gamma|_H$).
%Therefore it suffices to show that $V\ominus\epsilon$ satisfies $(\ast)$ when $H=G$. 
\end{proof}

%\begin{lemma}[Obsolete]
%\label{lemGRgamma}
%Let $F_\s$ be the quadratic extension of $K(\s_0)^\nr$. Then
%$$
%  G=\Stab_{I_K}(\s)/I_{F_\s}, \quad R=\s_0, \quad\text{and}\quad \gamma=\chi_\s
%$$
%satisfy the assumptions of Theorem \ref{abrep}.
%\end{lemma}
%
%\begin{proof}
%Lemma \ref{autA1}, plus what we know about $\gamma$.
%\end{proof}
%
%\begin{lemma}[Obsolete]
%For every $H\<G$ we have $f(H)=2\genus(Z_\s/H)$.
%\end{lemma}
%
%\begin{proof}
%By appendix.
%
%$$
%\begin{array}{ll}
%%\hline
%\text{Case} & \text{Genus of $Z/H$} \cr
%\hline
%\kappa\in H & 0 \cr 
%\kappa\notin H,\>\> |R|=rp^km     & \lfloor (r-1)/2\rfloor \cr
%\kappa\notin H,\>\> |R|=rp^km+p^k & \lfloor r/2\rfloor \cr
%%\hline
%\end{array}
%$$
%\end{proof}
%
%[Obsolete]
%Applying the lemmas and Theorem \ref{abreconstr} we get Theorem \ref{abmain}.
%(And observe that everything indeed factors through $G$ of Lemma \ref{lemGRgamma}.)
%

%\section{Representation-theoretic computation}

In the remainder of this section we prove the ingredients of Theorem \ref{abrep}.

\begin{notation}
Let $G$, $\Gbar=T\rtimes \Cy m$, $R$, $\gamma$, $\kappa$, $f$, $V$ and $\epsilon$ 
be as in Theorem~\ref{abrep}.
Write $r$ for the number of regular orbits of $\Gbar$ on $R$,
and $\tilde r=0$ if $r$ is even and $1$ if $r$ is odd 
(i.e. the last binary digit of $r$).
\end{notation}

\begin{lemma}
\label{abR1}
We have
\begin{enumerate}
\item 
$(\C[\Gbar/\Cy m]\ominus\triv)^{\oplus m}\iso \C[\Gbar]\ominus\C[\Gbar/T]$ as representations of $\Gbar$.
\item
The representations $\C[\Gbar]$ and $\C[\Gbar/\Cy m]\ominus\triv$ of $\Gbar$
are invariant under twisting by one-dimensional representations of $\Gbar/T\iso \Cy m$. 
\item
For every one-dimensional representation $\psi$ of $G/T$, the $G$-represen\-tations 
$\psi\tensor\C[\Gbar]$ and $\psi\tensor(\C[\Gbar/\Cy m]\ominus\triv)$ have
rational characters.
\end{enumerate}
\end{lemma}

\begin{proof}
(1)
The group $\Gbar/T\iso \Cy m$ acts on 1-dimensional characters of $T$ by conjugation,
and from the action of $\Cy m$ on $T$ we see that every non-trivial character has trivial stabiliser. 
Hence, by Clifford theory (or see 
\cite[\S8.2]{SerLi}), every irreducible representation of $\Gbar$ is either a lift of 
a 1-dimensional one of $\Gbar/T$, or is $m$-dimensional and is induced from 
a 1-dimensional representation of $T$. In particular, 
$\C[\Gbar]\ominus\C[\Gbar/T]$ is the sum of the $m$-dimensional irreducibles, 
each with multiplicity $m$.

Since $\Gbar\iso T\rtimes \Cy m$, as $T$-sets
$$
  \Res_T(\Gbar/\Cy m) \iso T,
$$ 
and so the representation $\Res_T(\C[\Gbar/\Cy m]\ominus\triv)$ 
contains every non-trivial one-dimensional representation of $T$. So, by Frobenius reciprocity,
$\C[\Gbar/\Cy m]\ominus\triv$ contains every $m$-dimensional irreducible of $\Gbar$. 
By comparing dimensions, each occurs with multiplicity one, and the claim follows.

(2) The twist-invariance is clear for $\C[\Gbar]$ and $\C[\Gbar/T]$, 
and so follows for $\C[\Gbar/\Cy m]\ominus\triv$
from (1).

(3) Generally, for every rational character $\rho$ of $G$ which is invariant 
under $\Cy m$-twists, $\psi\tensor\rho$ is again rational. Indeed, $\psi^2$ clearly kills 
$\kappa$ and factors through $\Cy m$,
and so
$$
  \psi(g)^2\rho(g)=\rho(g) \qquad\text{for all $g\in G$}.
$$
Thus, either $\psi(g)\rho(g)=0$ or $\psi(g)=\pm 1$, and 
$\psi(g)\rho(g)$ is rational in both cases.
\end{proof}

\begin{lemma}
\label{abR2}
%Write $\Gbar=G/\langle\kappa\rangle\iso T\rtimes \Cy m$.
We also have
\begin{enumerate}
\item
Either
\begin{itemize}
\item[(a)]
$|R|\equiv 0\mod m$ and $\gamma^2=\triv$;
as a $\Gbar$-set, $R$ is a union of regular orbits; or
%\end{itemize}
%\noindent or
%\begin{itemize}
\item[(b)]
$|R|\equiv 1\mod m$ and $\gamma^2=\alpha\ne \triv$;
as a $\Gbar$-set, $R$ is a union of regular orbits and
one non-regular orbit $\iso\Gbar/\Cy m$.
%, which is the set of fixed points of non-trivial 
%elements of $\Gbar$ on~$R$.
\end{itemize}
\item
The representations $V$ and $V\ominus\epsilon$ have rational characters.
\end{enumerate}
\end{lemma}

\begin{proof}
(1) 
%The assumption that non-trivial elements of $T$ act fixed point free on $R$ implies that
%$R$ is a union of regular $T$-orbits. 
%
%Suppose $\Gbar$ has a non-regular orbit on $R$.  
%It is a union of $m/k$ regular $T$-orbits for some $k>1$, 
%which $\Cy m$ permutes transitively. 
%As $k$ and $|T|$ are coprime, elements of $C_k$ must have fixed points in each one of them.
%But they have at most one fixed point in total, and so $k=m$, such an orbit
%is $\iso \Gbar/\Cy m$, and it is necessarily unique. Hence
%$$
%  R \iso \Gbar \amalg \cdots \amalg \Gbar 
%    \qquad\text{or}\qquad
%  R \iso \Gbar \amalg \cdots \amalg \Gbar \amalg \Gbar/\Cy m.
%$$
%In the first case $|R|\equiv 0\mod m$. In the second case, $R$ is (by the above argument
%with $\Gbar=\Cy m$) a union of regular orbits plus a singleton, and so $|R|\equiv 1\mod m$.
The $\Gbar$-structure of $R$ follows from Lemma \ref{groupstr}.
By assumption, $\gamma^2=\alpha^{|R|}$, and this is $\triv$ or $\alpha$, 
respectively.

(2) In case 1(a), the representation $\C[R]\ominus\triv$ is realisable over $\Q$ and $\gamma$ has 
order $\le 2$, so $V=\gamma\otimes (\C[R]\ominus\triv)$ is also realisable over $\Q$. In case 1(b),
$V=\gamma\otimes (\C[R]\ominus\triv)$ has rational character by Lemma \ref{abR1} (3).
Finally, $\epsilon=0$ or $\epsilon=\det V$, each of which is also rational.
\end{proof}

\begin{lemma}[Case \I] 
\label{lemcaseI}
If $\kappa\ne \id$, then
$$
  \dim (\gamma\tensor\C[R])^G = \dim \gamma^G = \dim \epsilon^G = 0.
$$
In particular, $f(G)=\dim (V\ominus\epsilon)^G$ in this case.
\end{lemma}

\begin{proof}
Because $\kappa\ne \id$, $\gamma(\kappa)=-1$. Since $\kappa$ acts trivially on $\C[R]$, 
both $\gamma$ and $\gamma \otimes \C[R]$ have trivial $G$-invariants, as
does $\epsilon=\gamma^{|R|-1}\det\C[R]$ when $|R|$ is even.
%
%$\Res_{\langle\kappa\rangle} (\gamma \otimes \C[R]) = 
%  \upsilon^{\bigoplus |R|}$, 
%  where $\upsilon$ is the non-trivial character of $\Cy2=\langle\kappa\rangle$.
%In particular, $\gamma \otimes \C[R]$ has no $H$-invariants.
\end{proof}

%From now on suppose that we are in Case \II{} or Case \III,
%$\kappa=1$ and $H=G=T\rtimes \Cy m$.

%
%Observe that 
%$R=R_r$ or $R_r\coprod R_i$ with $R_r$ union of regular $G$-orbits,
%and $R_i=G/\Cy m$.
%And also, if $R=R_r$, then $\gamma$ has order $\le 2$.
%In Case \II, $R=R_r$.
%
%
%Note that $R_i$ as a $\Cy m$-set is a singleton plus a union of regular orbits. 
%In particular $R$, as a
%$\Cy m$-set, is a union of regular orbits, with possibly an extra singleton. 
%(And $|R|$ is 0 or 1 mod m.)
%
%Note that $\C[R_r]$ and $\C[R_i]\ominus \triv$ are invariant under $\Cy m$-twists.
%
%We have $\det V=\gamma^{|R|-1}\det\C[R]$, and both $\gamma^{|R|-1}$ and $\det\C[R]$ 
%are either $\triv$ or $\eta$ (sign character). Note also that $V$ is a 
%rationally traced representation, which is one of our claims.
%
%

%\begin{lemma} 
%We have
%$$
%\begin{array}{lllllllc}
%%  &&f(H) = & 0 
%%     && \text{if} \>\> \kappa\in H, \gamma(\kappa) =-1 
%%    && (\text{\rm Case \I}) \cr
%  &&f(H) = & r - 2 + \tilde r
%     && \text{if} \>\> \gamma=\triv  && (\text{\rm Case \II})\cr
%  &&f(H) = & r - \tilde r
%     && \text{if} \>\> \gamma\ne\triv && (\text{\rm Case \III}).\cr
%\end{array}
%$$
%\end{lemma}

\begin{lemma}[Case \II] 
\label{lemcaseII}
If $\kappa=\id$ and $\gamma=\triv$, then
$$
  f(G) =  r - 2 + \tilde r, \>\> 
  \dim (\gamma\tensor\C[R])^G = r, \>\> 
  \dim \gamma^G = 1 \>\>\text{and}\>\>
  \dim \epsilon^G = 1-\tilde r.
$$
In particular, $f(G)=\dim (V\ominus\epsilon)^G$ in this case.
\end{lemma}

\begin{proof}
First of all, $G=\Gbar$ since $\kappa=\id$. Next, as $\gamma=\triv$, 
$R$ is a union of $G$-regular orbits by Lemma \ref{abR2}(1). Now,
$$
\begin{array}{llllll}
  \displaystyle
  f(G) = 2\Bigl\lfloor \frac{r|G|}{2|G|} - \frac12 \Bigr\rfloor = 
         2\Bigl\lfloor \frac{r-1}{2} \Bigr\rfloor = \leftchoice
           {r-2}{\text{if $2\,\,|\,r$}}{r-1}{\text{if $2\nmid r$}} = r-2+\tilde r,\\[3pt]
  \dim(\gamma \otimes \C[R])^G=\dim\C[R]^G=|R/G|=r, \\[3pt]
  \dim \gamma^G = 1.
\end{array}
$$
If $|R|$ is even, then 
$$
  \epsilon = \det V = \gamma^{|R|-1}\det\C[R] = \det\C[R]=(\det\C[G])^{r}.
$$
This is non-trivial if and only if $r$ is odd%
\footnote{For a group $G$, recall that $\det\C[G]$ is the trivial character 
unless $G$ has a non-trivial cyclic 2-Sylow subgroup, in which case $\det\C[G]$ is of order 2.}
(in which case $|G|$ is even as $|R|$ is even).
So $\dim\epsilon^G= 1 - \tilde r$, as claimed.

On the other hand, if $|R|$ is odd, then $\epsilon=0$. As $R$ is a union of 
regular $G$-orbits, there must be an odd number of them, so that $\tilde r=1$ and once again
$\dim\epsilon^G= 1 - \tilde r$.
%
%We have $\det V=\gamma^{|R|-1}\det\C[R]$, and both $\gamma^{|R|-1}$ and $\det\C[R]$ 
%are either $\triv$ or $\eta$. Suppose furthermore that $|R|$ is even, so that 
%$\epsilon=\det V$.
%
%$\det \C[I_x]$ is $\triv$ when $m$ is odd, and $\eta$ when $m$ is even.
%
%$\det \C[R_r]$ is $\eta$ when $m$ is even and $a$ is odd, and $\triv$ otherwise.
%
%Suppose $R=R_r$. In this case, $\gamma$ has order at most two, $\gamma^{|R|-1}=\gamma$,
%and $|R|=am\cdot$odd. So $\det\C[R]=\det \C[R_r]=\eta^a$, and $\det V=\gamma\eta^a$.
%
%We have $\gamma|_H=\triv$. 
%
%If $a$ is even, then $\det V=\eta^\gamma$ is trivial on $G$ as well.
%Also $R/G$ is $a$ copies of $\C[I_x]/G$, and has even number of fixed points.
%
%Now suppose $a$ is odd and $m$ is even. Here $\det\C[R]=\det\C[I_x]^a=\eta^a=\eta$. 
%Its restriction to $G$ is non-trivial if and only if $G$ contains a 2-Sylow subgroup of $G$,
%equivalently $[G:G]$ is odd,
%equivalently $I_x/G\langle\kappa\rangle$ has odd cardinality, 
%equivalently $|\C[I_x]/G|$ is odd, equivalently $|R/G|$ is odd.
%So $(\det V)^G=0$ if and only if $|R/G|$ is odd, as claimed.
%Now suppose $|R|$ is odd, so $\epsilon=0$. Suppose furthermore that $R=R_r$.
%In this case $a$ and $m$ are both odd. 
%Again, the action on $R$ is through a group of odd order, so 
%$$
%  |R/G|\equiv |R|=|R|\equiv 1\mod 2.
%$$
\end{proof}

\begin{lemma}[Case \III] 
\label{lemcaseIII}
Suppose $\kappa=\id$ and $\gamma\ne \triv$. Then
\begin{enumerate}
\item
$f(G) = r - \tilde r$.
\item
$R$ is a union of regular $G$-orbits if and only if $|R|$ and $m$ are both even.
\item 
We have
$$
  \dim (\gamma\tensor\C[R])^G = r, \quad 
  \dim \gamma^G = 0 \quad\text{and}\quad
  \dim \epsilon^G = \tilde r.
$$
In particular, $f(G)=\dim (V\ominus\epsilon)^G$ in this case.
\end{enumerate}
\end{lemma}

\begin{proof}
First of all, $G=\Gbar=T\rtimes \Cy m,$ since $\kappa=\id$, and $m>1$ as $\gamma\ne\triv$. 
By Lemma \ref{abR2}(1), $R$ decomposes as a $G$-set as
$$
  R = G^{\amalg r} \qquad \text{or} \qquad R = G^{\amalg r}\amalg G/\Cy m.
$$

(1) By definition of $f$,
$$
  f(G) = 2\Bigl\lfloor \frac{|R|}{2|G|} \Bigr\rfloor 
       = 2\Bigl\lfloor \frac{r|G|+\delta}{2|G|} \Bigr\rfloor,
%
%         2\Bigl\lfloor \frac{r-1}{2} \Bigr\rfloor = \leftchoice
%           {r-2}{\text{$r$ even}}{r-1}{\text{$r$ odd}} = r-2+\tilde r,\\[3pt]
$$
with $\delta=0$ or $\delta=|G/\Cy m|$. As $\delta<|G|$, we have
$
  f(G) = 2\bigl\lfloor\frac r2\bigr\rfloor=r-\tilde r.
$

(2) 
If $m$ is odd, then $\gamma\ne\triv\Rightarrow\gamma^2\ne\triv$, so $R$ has an irregular orbit
by Lemma \ref{abR2}(1b). 
%
%If $R$ is a union of regular orbits, then $|R|$ is a multiple of $m$.
%As $G$ has odd order, we have
%$$
%  \gamma^2=\alpha^{|R|}=\triv \quad \implies \quad \gamma=\triv,
%$$
%contradiction. 
%
If $m$ is even, then every regular orbit has even size while $|G/\Cy m|=|T|$ is odd,
so the parity of $|R|$ is determined by whether there is an irregular orbit.

(3)
Clearly $\dim\gamma^G=0$. 
By Lemma \ref{abR2},
$$
  \gamma\otimes\C[G]\iso \C[G] \quad\text{and}\quad \gamma\tensor(\C[G/\Cy m]\ominus\triv)=\C[G/\Cy m]\ominus\triv,
$$
and it follows that either
$$
  \gamma\otimes\C[R] = \C[G]^{r} \quad \text{or} \quad \gamma\otimes\C[R] = \C[G]^{r}\oplus \C[G/\Cy m]\oplus\gamma\ominus\triv,
$$
depending on whether $R$ is a union of regular orbits or not.
Each $\C[G]$ summand has 1-dimensional $G$-invariants, and their dimensions add up to~$r$,
while
$$
  \dim(\C[G/\Cy m])^G  + \dim\gamma^G-\dim\triv^G
    = 1 + 0 - 1 = 0.
$$
This proves the first claim, and it remains to show that $\dim\epsilon^G=\tilde r$.

Suppose $|R|$ is even, so that $\epsilon = \det V$.
If $m$ is odd, then $G$ has odd order while $\det V$ has rational character by Lemma \ref{abR2}(2),
so $\epsilon=\det V=\triv$.
On the other hand, $R$ has an irregular orbit by (2), and all orbits are of odd size, 
so $r$ is odd. Hence $\tilde r=1=\dim \epsilon^G$. 

If $m$ is even, then $R$ is a union of regular orbits by (2), and
$\det \C[G]=\eta$, the non-trivial character of $\Cy m$ of order 2.
Moreover, $\gamma=\eta$ because $\gamma\ne\triv$ but $\gamma^2=\triv$ by
Lemma \ref{abR2}(1). Therefore
$$
\begin{array}{llllll}
  \epsilon = \det V\!\! &=\!\!& \det(\gamma\otimes(\C[R]\ominus\triv)) =  
    \gamma^{-1}\otimes\det(\gamma\otimes\C[G]^{\oplus r}) \cr
    &=\!\!& \gamma^{-1}\otimes\det(\C[G]^{\oplus r}) = \eta^{r-1},
\end{array}
$$
and so $\dim \epsilon^G = \tilde r$.

Finally suppose $|R|$ is odd, so that $\epsilon=0$ and we need to show that $r$ is even.
By (2), $R$ has an irregular orbit, so
$\gamma^2=\alpha$ by Lemma \ref{abR2}(1), which has
order $m$. Hence $m$ must be odd, as $G$ has no 1-dimensional representation of order $2m$.
Thus every $G$-orbit has odd size, and $r\equiv |R|-1\mod 2$ is even.
\end{proof}

\begin{lemma}
\label{EGA}
If $R$ is even, then $\epsilon=\alpha^{|R|/2}\tensor\gamma^{-1}$.
%, and
%$\epsilon(g)=-1$ if $g$ swaps the two points at infinity of $C$, and $+1$ otherwise.
\end{lemma}

\begin{proof}
%The second claim is proved in Proposition \ref{autC} (6). 
%In particular, 
%$\alpha^{|R|/2}\tensor\gamma^{-1}$
%is a character of order 2, and so is $\epsilon$ by Lemma \ref{abR2}. 
%The first claim is that 
We need show that
$
  \gamma^{|R|-1}\det(\C[R]) = \alpha^{|R|/2}\tensor\gamma^{-1}.
$
%
%is equivalent to $\gamma^{|R|}\det(\C[R]) = \alpha^{|R|/2}$. 
As $R$ is even and 
$\gamma^2=\alpha^{|R|}$ by assumption (cf. Thm. \ref{abrep}), 
%$\gamma^{|R|}\det(\C[R]) = \alpha^{|R||R|/2}$, and so we need to show that
this is equivalent to
$$
  \det(\C[R]) = \alpha^{(|R|-1)|R|/2}.
$$
Both sides are rational characters (that is of order 1 or 2) of $\bar G$; 
this is clear for $\det(\C[R])$, and follows from the fact that 
$|R|\equiv 0,1\mod m$ for the right-hand side (Lemma \ref{abR2}), and
$m$ is the order $\alpha$. Moreover, if $|R|\equiv 1\mod m$ then $m$ is odd as $R$ is even,
so both characters are trivial. Therefore we may assume that $R=\C[\bar G]^{\oplus r}$,
a union of $r$ regular orbits (Lemma \ref{abR2} again). If $r$ is even, then the 
left-hand side is trivial, and so is the right-hand side, as 
$|R|/2$ is a multiple of $m$. 

Finally, suppose $r$ is odd and $|R|=rm|T|$, in particular $m$ is even. Let $\eta=\alpha^{m/2}$ 
be the non-trivial character of order 2 of $\bar G$. Both
$\det(\C[R])$ and $\alpha^{(|R|-1)|R|/2}$ are odd powers of $\eta$ in this case 
(see footnote above), and the claim follows.
\end{proof}

\section{Descending morphisms}
%%%%%%%%%%%%%%%%%%%%%%%%%%%%%%%%%%%%%%%%%%
\label{sfrob}

In this section we describe how certain morphisms descend to quotients of hyperelliptic curves.
Our motivation comes from the arithmetic of hyperelliptic curves over finite and local fields, 
and the question of how the Frobenius automorphism acts on the quotient curve. 
See \S\ref{sex} for an example.

Let $k$ be a field of characteristic $p\!>\!2$, and let $C/k$ be a hyperelliptic curve.
%$$
%  C: Y^2 = c \prod_{r\in R}(X-r) \quad \qquad (R\subset\bar k),
%$$
Let $G\subset \Aut_k C$ be an affine group of automorphisms, given by 
$$
  g(X) = \alpha(g) X+\beta(g), \qquad g(Y)=\gamma(g)\,Y \qquad \qquad (g\in G)
$$
as before, and $C/G$ be the quotient curve given explicitly in Theorem \ref{appmain}.
%Recall that a morphism $\Frob: C\to C$ is \emph{$q$-semi-linear} if it acts as 
%\marginpar{NOT supposed to be semi-linear, just linear like the classical Frobenius}
%$z\mapsto z^q$ on $k\subset k(C)$ for some power $q$ of $p$. 
We say 
that a morphism $\Frob: C\to C$ \emph{normalises} $G$ if for every $g\in G$ there is $g'\in G$ for which
$g\Frob =\Frob g'$.

\begin{theorem}
\label{appfrob}
Suppose $\vchar k=p>2$, $G$ does not
contain the hyperelliptic involution, and
$\Frob: C\to C$ is a morphism of the form 
$$ 
  \Frob\,X=aX^q+b,\quad \Frob\,Y=d\>Y^q  %\qquad \qquad\text{($q$ power of $p$)}
$$
that normalises $G$, with $q$ a power of $p$.
Then $\Frob$ descends\footnote{that is $\pi\Frob=\Psi\pi$ where $\pi: C\to C/G$ is the quotient map} 
to a morphism $\Psi: C/G\to C/G$
given~by
%given in the notation of Theorem \ref{iappmain} by
%\marginpar{No $q$th powers in $\Psi(x)$?}
$$
\begin{array}{llllll}
  \Psi\,x &=& \displaystyle a^{|G|} x^q + \prod_{g\in G} (\alpha(g)b+\beta(g)),\cr
  \Psi\,y &=& \displaystyle \bigleftchoice
     {d\>y^q}{\text{if $\gamma=\triv$,}}
     {a^{\lfloor m/2\rfloor |G|/m}\> d\>y^q}{\text{if $\gamma\ne\triv$,}}
\end{array}
$$
where $m$ is the prime-to-$p$ part of $|G|$, and $x$, $y$ and the model for $C/G$ 
are those of Theorem \ref{appmain}.
\end{theorem}

%\begin{theorem}
%\label{appfrob}
%
%Suppose $\vchar k=p>0$ and $C/k$ is a hyperelliptic curve with $G\subset\Aut C$ preserving the divisor at infinity and
%%as in Theorem \ref{appmain},
%not containing the hyperelliptic involution.
%Suppose $\Frob: C\lar C$ is a semi-linear\footnote{i.e. $z\mapsto z^q$ on $k\subset k(C)$} 
%morphism of the form 
%$$ 
%  \Frob\cdot X=aX^q+b,\quad \Frob\cdot Y=d\>Y^q  \qquad \qquad\text{($q$ power of $p$)}
%$$
%that commutes with the projection $C\lar C/G$.
%\marginpar{Is this enough for $\Frob\cdot g=g'\cdot \Frob$?}
%It descends to a morphism 
%$C/G\lar C/G$, %given in the notation of Theorem \ref{appmain} by
%$$
%\begin{array}{lll}
%  \Frob\cdot x = a^{|G|} x^q + I^{(q)}(b), \\%\qquad 
%  \Frob\cdot y = \bigleftchoice
%     {d\>y^q}{\text{if G acts trivially on the Y-coordinate,}}%{\text{in Case 2 of \ref{appmain},}}
%     {a^{\lfloor m/2\rfloor |G|/m}\> d\>y^q}{\text{otherwise,}}
%\end{array}
%$$
%where $I(X)=\prod_{g\in G}g(X)$ and $m$ is the order of the prime-to-$p$ part of $G$.\marginpar{standardise $I^{(q)}$ notation}
%%
%%d\>y^q.
%%$$
%%in case 2 and 
%%$$
%%  \Frob\cdot x = a^{|G|} x^q + I^{(q)}(b), \qquad 
%%  \Frob\cdot y = a^{\lfloor m/2\rfloor |G|/m}\> d\>y^q.
%%$$
%%in case 3.
%\end{theorem}

\begin{proof}
%Here $\vchar k=p>0$, $q$ is a power of $p$, and $k$ is perfect. 
%
%Assume $G$ does not contain the hyperelliptic involution.
We may assume that $k$ is algebraically closed.
The quotient map $C\to C/G$ 
%is given by $(X,Y)\mapsto (I(X),J(X)Y)$,
%and 
corresponds to a field inclusion
$$
  k(C/G) = k(x,y) \quad \injects \quad  k(X,Y) = k(C). 
$$
%Frobenius acts by some semi-linear\footnote{i.e. $x\mapsto x^q$ on $k$} map
%$$ 
%  \Frob\cdot X=aX^q+b,\qquad \Frob\cdot Y=d\>Y^q,
%$$
%preserving $k(C/G)$ (as it normalizes $G$). 
The morphism $\Frob$ preserves $k(C/G)=k(C)^G$ as it normalises $G$, 
so it descends to a morphism $\Psi: C/G\to C/G$. On the level of functions,
$\Psi$ is just the restriction of $\Phi$ to $k(C/G)$.
We now describe the action of $\Psi$ on the generators $x$ and $y$ explicitly.
Note that for every polynomial $h(X)$,
%\marginpar{$h\mapsto h^{(1/q)}$ in RHS?}
$$
  \Frob\cdot h(X) \!=\! h(aX^q+b) \!=\! 
    h((\sqrt[q]{a}X+\sqrt[q]{b})^q) 
    \!=\! (h^{(1/q)}(\sqrt[q]{a}X+\sqrt[q]{b}))^q,
  \eqno(\dagger)
$$
where $h^{(1/q)}(X)$ denotes the polynomial obtained from $h(X)$ by raising every coefficient to 
the power $1/q$.

%
%The stabiliser of $C_{\alpha\beta}$ in $G$ acts on $k(X,Y)$ by semi-linear
%automorphisms (with inertia acting $k$-linearly). Since inertia forms
%a normal subgroup, it follows that $\Frob_q$ preserves inertia invariants $k(I,Y)$.

\smallskip
\noindent
{\bf Action on $x=I(X)$.}
%By $(\dagger)$, 
%we have
%It is clear that $\Frob\cdot I(X)\in k(X)$ is $G$-invariant. 
%, so it is a polynomial
%in $I(X)$, namely
%\marginpar{check $(q)$ in $I$} 
%$$
%  \Frob\cdot I(X) = I^{(q)}(aX^q+b) = 
%    I^{(q)}((a^{1/q}X+b^{1/q})^q) 
%    = (I(a^{1/q}X+b^{1/q}))^q.
%$$
Recall from Theorem \ref{appmain} that $x=I(X)=\prod_{g\in G}g(X)$.
Since it is $G$-invariant, so is 
%\marginpar{$I\mapsto I^{(1/q)}$ in RHS?}
$$
  \Frob\cdot I(X) = (I^{(1/q)}(\sqrt[q]{a}X+\sqrt[q]{b}))^q.
$$
This has a unique $q$th root, namely $I^{(1/q)}(\sqrt[q]{a}X+\sqrt[q]{b})$, which must 
therefore be $G$-invariant as well.
%
%and therefore also $\Frob\cdot I(X)\in k(X)$ are $G$-invariant, 
%$G$-action on $k(X,Y)$
%commutes with $q$th powers, and the $q$th power map $k(X,Y)\to k(X,Y)$ is 
%injective, the polynomial $I(a^{1/q}X+b^{1/q})$ must itself be $G$-invariant.
As it has the same degree as $I(X)$, by Lemma \ref{autA1}(5),
%In other words, it is a polynomial in $I(X)$, by Lemma \ref{autA1}(5). 
%But they have the same degree, so 
$$
  I^{(1/q)}(\sqrt[q]{a}X+\sqrt[q]{b}) = u\, I(X)+v 
$$
%it must be of the form $c I(X)+d$ for some $c,d\in k$. 
for some $u,v\in k$.
Comparing the leading
and the constant coefficients, we see that $u=a^{|G|/q}$ and $v=I^{(1/q)}(\sqrt[q]{b})$. Thus
%\marginpar{$I^{(q)}(b)\mapsto I(b)$ in RHS?}
$$
  \Psi(x) = (ux+v)^q = a^{|G|} x^q + I(b).
$$
%$$
%  (X,Y)\quad \mapsto \quad (a^{|U|} X^q + I^{(q)}(b), d Y^q).
%$$

\smallskip
\noindent
{\bf Action on $y$.} We have two cases:
%\marginpar{Check computati\-ons for $y$}

{\bf Case $\gamma=\triv$.} Here $y=Y$, and so
%the action of $\Psi$ on it is the same as that of $\Frob$,
%
%Here the action on the $y$-coordinate of $C/G$ is the same as on $Y$,
$$
  \Psi(y) = d\>y^q.
$$

{\bf Case $\gamma\ne\triv$.}
In this case $y=(I_T(X)\!-\!\lambda)^{\lfloor m/2\rfloor}Y$,
where $I_T(X)-\lambda$ is the unique monic polynomial of degree $|G|/m$ 
that is $G$-invariant up to scalars (see Lemma \ref{autA1} (7)).
%
%
%.
%From the proof of Theorem \ref{appmain}(3) ``Case $|T|\ne 1$'', 
%its roots form the unique non-regular orbit of $G$ on $\A^1_k$, via the $X$-coordinate action.
%Therefore
%$I_T(X)\!-\!\lambda$ is, up to scalars, the unique polynomial of degree $|G|/m$ 
%that is $G$-invariant up to scalars. 
Because $\Frob$ normalises $G$,
$$
  g\cdot\Frob\cdot(I_T(X)-\lambda)=\Frob\cdot g'\cdot(I_T(X)-\lambda) 
    = \text{scalar}\cdot \Frob\cdot(I_T(X)-\lambda),
$$
so $\Frob\cdot(I_T(X)-\lambda)$ is also $G$-invariant up to scalars.
By $(\dagger)$, it is a $q$th power,
$$
  \Frob\cdot(I_T(X)-\lambda) = \bigl(I_T^{(1/q)}(\sqrt[q]{a}X\!+\!\sqrt[q]{b})-\sqrt[q]{\lambda}\bigr)^q,
$$
and $I_T^{(1/q)}(\sqrt[q]{a}X\!+\!\sqrt[q]{b})-\sqrt[q]{\lambda}$ must be $G$-invariant up to scalars as well.
But it has the same degree as $I_T(X)$, so by uniqueness we must have
%has at most $|G|/m$ distinct roots, we must have 
$$
  \Frob\cdot(I_T(X)-\lambda) = \text{constant}\cdot (I_T(X)-\lambda)^q.
$$
Comparing the leading terms, we see that the constant is $a^{\deg I_T(X)}=a^{|G|/m}$.
Hence
%\vskip 3cm
%
%
%
%
%Here 
%$$
%  \Frob\cdot J(X) = \bigl(I_T^{(q)}(aX^q+b)-\lambda\bigr)^{\lfloor m/2\rfloor}.
%$$
%Because $\Frob$ preserves the unique irregular $G$-orbit, the right-hand side must 
%have the same roots(?) In which case,
$$
  \Psi\cdot y = 
    a^{\lfloor m/2\rfloor |G|/m}
    (I_T(X)-\lambda)^{q\lfloor m/2\rfloor} \>d\>Y^q =
    a^{\lfloor m/2\rfloor |G|/m}\>d\>y^q.
$$
%In other words, the action of $\Frob$ on the $y$-coordinate of $C/G$ is 
%$$
%  y \>\>\mapsto\>\> a^{\lfloor m/2\rfloor |G|/m}\> d\>y^q.
%$$
\end{proof}

%
%$$
%  k(B) = k(I,J) \quad \injects \quad  k(X,Y) = k(C_{\alpha\beta}),
%$$
%with the same $I$ as in the $Y$-invariant case, and $J=Y/\sqrt{I-\lambda}$ (as in ?). 
%Frobenius acts on $I$ as before, and 
%$$
%  \Frob\cdot J = \alpha J
%$$
%for some constant $\alpha$. INCOMPLETE.
%

%%%%%%%%%%%%%%%%%%%%%%%%%%%%%%%
\section{An example}
%%%%%%%%%%%%%%%%%%%%%%%%%%%%%%%
\label{sex}

To illustrate the results of this article, let us identify 
the local Galois representation attached to a specific hyperelliptic curve over a local field. 
We will explain how to use apply these methods in greater generality in a forthcoming work 
\cite{hyp, bwgen},
and for the moment will confine ourselves to one example.

Let $Z/\Q_3$ be the hyperelliptic curve of genus 3 given by
$$
  Z:\> y^2=x^8+3^4. %  \qquad   /K=\Q_3.
$$
Let $\zeta$ be a primitive 8th root of 1 and let $\alpha=\sqrt{3\zeta}$ be a root of $x^8+3^4$, so that the other roots are $\zeta^i\alpha$ for $i=1,...,7$. Let
%\marginpar{kill $K$?} $K=\Q_3$ and 
$F=K(\zeta,\alpha)$ be the splitting field of $x^8+3^4$. 
It is a $C_4$-extension of $\Q_3$ with ramification and residue degrees~2, 
so that in particular $\Q_3(\sqrt 3)^\nr=F^\nr$, the maximal unramified 
extension of $F$. Finally, let $\phi\in\Gal(F^\nr/\Q_3(\sqrt{3}))$ 
be the (arithmetic) Frobenius element and let $\tau\in\Gal(F^\nr/\Q_3^\nr)$ 
be the element of order~2. Thus $\phi$ gives a Frobenius element of $F^\nr/\Q_3$ 
and $\tau$ generates its inertia group.

We claim that the Galois action on $H^1(Z)=\H(Z_{\bar \Q_3},\Q_l)\otimes\C$ for $l\neq 3$ factors through $F^\nr/\Q_3$ and that, with respect to a suitable basis,
\begingroup\smaller[3]
$$
\phi^{-1}\mapsto
 \left( \begin{array}{cccccc}
 \sqrt{-3}\!\!\!\! & 0 &0 & 0 & 0 & 0  \\ 
0 & -\sqrt{-3}\!\!\!\! &0 & 0 & 0 & 0\\
0&0&1\!+\!\sqrt2\!\!\!\! &0 & 0 & 0  \\  
0&0&0 & 1\!-\!\sqrt2\!\!\!\!\! & 0 & 0  \\
0&0&0 &0 & -1\!+\!\sqrt2\!\!\!\!\! & 0   \\
0&0&0 &0 & 0 & -1\!-\!\sqrt2\!\!
\end{array} \right)\!\!, \quad%\quad
\tau\mapsto
 \left( \begin{array}{cccccc}
1 &0 & 0 & 0 & 0 & 0  \\  
0 & 1 & 0 & 0 & 0 & 0  \\
0 &0 & \rlap{$\!\!\!\text{-}1$} & 0 & 0 & 0  \\
0 &0 & 0 & \rlap{$\!\!\!\text{-}1$} & 0 & 0  \\ 
0 &0 & 0 & 0 & \rlap{$\!\!\!\text{-}1$} & 0  \\ 
0 &0 & 0 & 0 & 0 & \rlap{$\!\!\!\text{-}1$}
\end{array} \right).
$$
\endgroup
In particular, the representation is tamely ramified with conductor exponent~4 and local polynomial $1+3T^2$, so that the Euler factor at $p=3$ of the $L$-series of $Z/\Q$ is $\frac{1}{1+3^{1-2s}}$.

%Before proceeding with the proof, observe that $\tau$ and $\phi$ act as the permutations $(04)(15)(26)(37)$ and $(0145)(2763)$, respectively, on the roots of $x^8+3^4$, where the digit $k$ refers to the root $\zeta^k\alpha$. Indeed, writing $\zeta_{16}$ for the 16th root of unity with $\alpha=\zeta_{16}\sqrt3$, it is clear that as $\tau$ fixes $\Q_3^\nr$ it must fix both $\zeta_{16}$ and $\zeta$, and $\tau(\alpha)=\tau(\sqrt3)\tau(\zeta_{16})=-\alpha$. Similarly, $\phi(\sqrt3)=\sqrt3$ as $\phi$ fixes $L$, and $\phi(\zeta_{16})=\zeta_{16}^3$, as $\phi$ acts as the Frobenius automorphism on the residue field, so that $\phi(\zeta^k\alpha)=\zeta^{3k+1}\alpha$.

%To illustrate the results and their applications to hyperelliptic curves over local fields, we take a hyperelliptic curve of genus 3,
%$$
%  Z:\> y^2=x^8+3^4  \qquad   /K=\Q_3.
%$$
%We aim to determine $H^1(Z)=\H(Z_{\bar K},\Q_l)_\C$ as a $\GKK$-representation, and, in particular, its conductor exponent and the Euler factor. 
%The strategy is...

\subsection*{Galois action on the semistable model}

%The splitting field $F$ of $x^8+3^4$ is a $C_4$-extension of $K$,
\def\overundersym#1#2#3{
  \begin{array}{c}\scriptstyle #1\\[-4pt]#2\\[-4pt] \scriptstyle #3\end{array}
}
%$$
%  K=\Q_3
%  \quad \overundersym 2\subset{\text{\tiny unr.}} \quad
%  K(\zeta) 
%  \quad \overundersym 2\subset{\text{\tiny tot. ram.}} \quad
%  F=K(\zeta,\sqrt{3\zeta}), 
%%  \qquad \zeta=\text{ primitive 8th}
%$$
%where $\zeta$ is a primitive 8th root of 1. The roots are $\alpha=\sqrt{3\zeta}$ and 
%$\zeta^i\alpha$ ($i=1,...,7$). We pick a generator $\phi$ of $\Gal(F/K)$ that acts 
%as
%$$
%  \phi(\alpha)=\zeta\alpha, \qquad \phi(\zeta)=\zeta^3 
%\qquad \quad 
%  (\phi=(1256)(3874)\text{ on roots}).
%$$
%It is a Frobenius element of $\Gal(\bar K/K)$.

The curve $Z$ acquires good reduction over $F$, since the substitution
$$
%\begin{equation}\label{extr}
s(x)=\alpha x, \qquad s(y)=\alpha^4 y
%\end{equation}
$$
clearly transforms it to the model $\cC/O_F: y^2=x^8-1$, which has the 8th roots of unity as roots of the right-hand side.
We will write $C$ for its special fibre
$$
  C:\> y^2 = x^8-1  \qquad   {\text{over }}k=\F_9.
$$
%Write $\tau\in \Gal(F^\nr/K)$ for the non-trivial element in its inertia $(\iso \Cy2)$. 
%Thus, it maps $\sqrt{3}\mapsto-\sqrt{3}$, fixes the residue field, and acts as $\phi^2=(15)(26)(37)(48)$ on the 8 roots. 

The group
$$
  \Gal(F^\nr/\Q_3) = \langle \tau,\phi\rangle \iso \Cy2 \times \hat\Z
$$
acts naturally on $C(\bar k)$ (see \cite{bwgen}).
For an element $\sigma\in \Gal(F^\nr/\Q_3)$ this action is given by the composition
$$
  C(\bar k) 
  \overarrow{\text{lift}} 
  \cC(O_{F^\nr})
  \overarrow{s^{-1}} 
  Z(F^\nr)
  \overarrow{\sigma} 
  Z(F^\nr)
  \overarrow{s} 
  \cC(O_{F^\nr})
  \overarrow{\text{reduce}} 
  C(\bar k).
$$
In our example, %this action is%\marginpar{$X$ or $x$?}
% CHECKED TWICE!
$$
  \sigma: \quad (x,y)\quad\longmapsto\quad \biggl(\frac{\alpha^\sigma}{\alpha} x^\sigma,\Bigl(\frac{\alpha^\sigma}{\alpha}\Bigr)^4 y^\sigma\biggr)  \quad \mod \frak{m},
$$
where $\frak m$ is the maximal ideal of $F^\nr$.
%\marginpar{Observe: we have subtly replaced $\Phi$ viewed as a map of points to a 
%  morphism of curves that we are going to descend. TO DO: change $\Phi$ to $\phi$ as an 
%  element of Galois, and refer to [2] or some other paper for details.}
Observe\footnote{Writing $\zeta_{16}$ for the 16th root of unity $=\frac{\alpha}{\sqrt{3}}$, 
we clearly have $\alpha^\tau=-\sqrt{3}\zeta_{16}=-\alpha$ and 
$\alpha^\phi=\zeta_{16}^3\sqrt{3}=\zeta\alpha$, from the definitions of $\tau$ and $\phi$.}
that $\alpha^\tau=-\alpha$ and $\alpha^\phi=\zeta\alpha$,
so, in particular,
$$
  \tau: (x,y)\mapsto (-x,y) \qquad\text{and}\qquad
  \phi: (x,y)\mapsto (\bar\zeta^{-1} x^3,-y^3), 
$$
where $\bar\zeta$ denotes the image of $\zeta$ in $\F_9$.%\marginpar{$\bar\zeta^{-1}$?}

Define morphisms $g, \Phi: C\to C$ by the above formulae for $\tau, \phi$ on points. 
By the N\'eron-Ogg-Shafarevich criterion, the natural Galois action on the \'etale cohomology 
group $H^1(Z)$ factors through $\Gal(F^\nr/\Q_3)$. By \cite{bwgen}, 
there is an isomorphism of $\Q_l$-vector spaces
%\marginpar{the rest of Gal acts trivially?}
$$
  H^1(Z) \iso H^1(C), \eqno{(\ddagger)}
%  \qquad \text{as $\Gal(\bar\Q_3/\Q_3)$-modules}, 
$$
under which the action of $\tau$ and $\phi$ on $H^1(Z)$ translates to the natural 
geometric action of $g$ and $\Phi$ on $H^1(C)$. 
%where the action on the left is the natural Galois action on \'etale cohomology, and 
%where the action on $H^1(C)$ is through the semi-linear action of $\Gal(F^\nr/\Q_3)$ described above.
We are now in a position to apply the results of \S\ref{setalehyp} and \S\ref{sfrob} to explicitly determine the representation $H^1(Z)$.

\subsection*{$H^1(Z)$ as inertia representation}

To determine the action of the inertia group on $H^1(Z)$ we apply Theorem \ref{abmain} to the curve $C$
with the automorphism group $G=\langle g\rangle \iso \Cy2$ (with the action described above) and $\tilde\gamma=\triv$. 
%(as $g$ acts trivially on the $Y$-coordinate).
%To compute $H^1(Z)$ as an inertia representation, apply Theorem \ref{abrep} to the curve $C$ with the automorphism group   $G=\langle g\rangle \iso \Cy2$ coming from inertia and $\gamma=\triv$ (trivial action on $Y$).
Write $\eta$ for the non-trivial 1-dimensional representation of $G$. 
The roots of $x^8-1$ come in four regular $G$-orbits,
$$
  \{1,-1\}, \>\> \{\bar\zeta,-\bar\zeta\},\>\> \{\bar\zeta^2,-\bar\zeta^2\},\{\bar\zeta^3,-\bar\zeta^3\},
$$
so the theorem shows that, as a $G$-module,
$$
  H^1(Z) \iso H^1(C) \iso 
    \C[G]^{\oplus 4} \ominus \triv \ominus \triv \iso 
    \triv^{\oplus 2} \oplus \eta^{\oplus 4}.
$$
%As $\eta$ is tame, we deduce that, in particular,
%$$
%  \text{conductor exponent of $V$} = \codim H^1(Z)^{\Cy2} = 4.
%$$
In other words $g$ has eigenvalues $1$ and $-1$ with multiplicities $2$ and $4$, respectively, in its action on $H^1(Z)$, as claimed.

\subsection*{Counting fixed points}
To describe $H^1(Z)$ as a full $\Gal(\bar\Q_3/\Q_3)$-module, we will exploit the identifications
$$
  H^1(Z) \iso H^1(C) \qquad\text{and}\qquad
  H^1(Z)^{\Cy2} \iso 
  H^1(C)^{\Cy2} \iso 
    H^1(C/\Cy2).
$$
To be precise, first note that as $\Phi$ commutes with $g$, it preserves the
$\triv$- and the $\eta$-isotypical components of $\Cy2\!=\!\langle g\rangle$, 
and that $H^1(C)$ is completely determined by the 
%characteristic polynomial 
eigenvalues of $\Phi$ on them.
(The action of $\Phi$ is known to be semisimple, 
although in our case this will be clear as its eigenvalues will turn out to be distinct.)
By $(\ddagger)$,
the eigenvalues of $\phi$ on the inertia invariants $H^1(Z)^{\langle\tau\rangle}$ agree
with those of $\Phi$ on $H^1(C)^{\Cy2}$. These are, by Theorem \ref{thmcover} (2,3), 
the eigenvalues of $\Psi$ on $H^1(C/\Cy2)$,
where $\Psi$ is the induced morphism on $C/\Cy2$.

%The $\triv$-component %eigenspace 
%is directly related to the quotient $C/\Cy2$ through an isomorphism of 
%$\Gal(F^\nr/\Q_3(\sqrt 3))$-modules
%
%$$
%%  H^1(Z)^{\Cy2} \iso 
%  H^1(C)^{\Cy2} \iso 
%    H^1(C/\Cy2),
%$$
%and the action of $\Phi$ 
By Theorem \ref{appmain}, the quotient $C/\Cy 2$ is the genus 1 curve
$$
  C/\Cy2:\> y^2 = (x+1)(x+\bar\zeta^2)(x+\bar\zeta^4)(x+\bar\zeta^6)=x^4-1
$$
and, by Theorem \ref{appfrob}, $\Phi$ descends to 
$$
\begin{array}{cccccccccc}
  \Psi:     & C/\Cy2  & \lar         & C/\Cy2      \cr
           & (x,y)  & \longmapsto  & (-\bar\zeta^{2} x^3,-y^3).  \cr
\end{array}
%$$
%$$
%  (X,Y)\mapsto (\bar\zeta^2 X^3,-Y^3)  
$$
From the Lefschetz fixed point formula, the inverse characteristic polynomial 
of $\Psi$ on $H^1(C/\Cy2)$ is
$$
  \det\bigl(1-\Psi^{-1}T\bigm|H^1(C/\Cy2)\bigr) = 1 - a T + 3T^2
$$
for some $a\in\Z$,
and its value at $T=1$ is the number of fixed points of $\Psi$ on the curve.
To find it explicitly, first count $\bar\F_3$-solutions to the system
$$
  y^2=x^4-1, \quad x=-\bar\zeta^{2} x^3, \quad y=-y^3.
$$
Starting from the last equation, 
$$
\begin{array}{llllll}
  y=0               &\implies&  x^4=1,\> x=-\bar\zeta^{2} x^3               &\implies& \text{no solutions;} \cr
  y=\pm\bar\zeta^2  &\implies&  x^4=\bar\zeta^4+1=0,\> x=-\bar\zeta^2 x^3 &\implies& x=0,\> y=\pm\bar\zeta^2. \cr
%  y=-\bar\zeta^2 &\implies&  x^4=\bar\zeta^2, x=\bar\zeta^2 x^3  &\implies& \text{no solutions} \cr
\end{array}
$$
Finally, to see the action on the points at infinity $\infty_\pm$, let  $s=\frac1x, t=\frac{y}{x^2}$. The
equation of the curve becomes%\marginpar{$1-t^4?$}
$$
  t^2 = 1-s^4, \qquad \infty_\pm=(0,\pm 1),
$$
and the transformation $\Psi: (x,y) \mapsto (-\bar\zeta^2 x^3,-y^3)$ on this chart is
$$
  s=\frac{1}{x}\mapsto \frac{-1}{\bar\zeta^2 x^3} = \bar\zeta^{2} s^3,  
  \qquad
  t=\frac{y}{x^2}\mapsto \frac{-y^3}{(-\bar\zeta^2 x^3)^2} = \frac{y^3}{x^6} = t^3.
$$
It fixes both $\infty_+$ and $\infty_-$. Overall, $\Psi$ has 4 fixed points, and
its inverse characteristic polynomial on $H^1(C/\Cy2)$ is therefore\footnote{Alternatively, $Z$ has 
visibly good reduction over $\Q_3(\sqrt 3)$, and if we pick a Frobenius element that 
fixes $\sqrt 3$, we end up counting points of the standard Frobenius 
$x\mapsto x^3, y\mapsto y^3$ on $(y^2=x^8+1)/\Cy2\iso y^2=x^4+1$. This has
four fixed points $(0,\pm 1)$ and $\infty_\pm$, recovering the same characteristic
polynomial. Similarly on the full curve $y^2=x^8+1$,
\begin{center}
\tt EulerFactor(Jacobian((HyperellipticCurve([GF(3)|1,0,0,0,0,0,0,0,1]))));
\end{center}
in Magma \cite{Magma} produces $(1+3T^2)(1+2T^2+9T^4)$, as before.}
$
  1+3 T^2.  
$
Hence the eigenvalues of $\Phi$ on this subspace are $\pm\frac{1}{\sqrt{-3}}$, as claimed.

Similarly, counting $a_i$ = the number of fixed points of $\Phi^i$ on $C$ itself, 
we find the sequence to be 
$(4,20,28,92,244,692,...)$. Thus, by the Lefschetz fixed point formula again, %\marginpar{+degree of zeta-fn?}
the inverse characteristic polynomial of $\Phi$ on the full space $H^1(C)$ is
$$
  \exp(\sum_{i\ge 1}\frac{a_i}iT^i)(1-T)(1-3T)  = (1+3T^2)(1+2T^2+9T^4).
$$
The first factor lives, as we have seen, on the $\triv$-component 
of $\Cy2\!=\!\langle g\rangle$, and so the second factor lives on the $\eta$-component. 
In other words, the eigenvalues of $\Phi$ on the $-1$-eigenspace of $g$ on $H^1(C)$ 
are $\pm\frac1{1\pm\sqrt2}$, as claimed.

%\begin{example}[Elliptic curves]
%Classification?
%\end{example}
%
%\begin{example}[Genus 2]
%Classification?
%\end{example}
%
%
%\begin{example}
%$y^2=x^p-x$ in characteristic $p$, with $G=C_p\rtimes C_{2(p-1)}$ and $p-1$-dimensional
%$H^1$.
%\end{example}

\begin{acknowledgements}
%We would like to thank ...
The second author is supported by a Royal Society University Research Fellowship.
This research is partially supported by \hbox{EPSRC} grants EP/M016838/1 and EP/M016846/1 	 
`Arithmetic of hyperelliptic curves'.
\end{acknowledgements}

%%%%%%%%%%%%%%%%%%%%%%%%%%%%%%%
% Bibliography
%%%%%%%%%%%%%%%%%%%%%%%%%%%%%%%

\end{document}